\documentclass[11pt]{amsart}
\usepackage[left=2.2cm,tmargin=2.5cm,bmargin=2.5cm,right=2.2cm]{geometry}
\usepackage{amsfonts}
\usepackage{amssymb}
\usepackage{hyperref}
\usepackage{stmaryrd}
\usepackage{mathrsfs}
\usepackage[all]{xy}
\usepackage{aliascnt}
\usepackage[capitalise]{cleveref}
\usepackage{amscd}
\usepackage{physics}
\usepackage{enumitem}
\usepackage{graphicx}
\usepackage[utf8]{inputenc}
\usepackage{pdfpages}

\theoremstyle{plain}
\newtheorem{thm}{Theorem}[section]

\newaliascnt{lem}{thm}
\newtheorem{lem}[lem]{Lemma}
\aliascntresetthe{lem}
\crefname{lem}{Lemma}{Lemmas}

\newaliascnt{cor}{thm}
\newtheorem{cor}[cor]{Corollary}
\aliascntresetthe{cor}
\crefname{cor}{Corollary}{Corollaries}

\newaliascnt{prop}{thm}
\newtheorem{prop}[prop]{Proposition}
\aliascntresetthe{prop}
\crefname{prop}{Proposition}{Propositions}

\newaliascnt{conj}{thm}

\aliascntresetthe{conj}
\crefname{conj}{Conjecture}{Conjectures}

\theoremstyle{definition}
\newaliascnt{defn}{thm}
\newtheorem{defn}[defn]{Definition}
\aliascntresetthe{defn}
\crefname{defn}{Definition}{Definitions}

\newaliascnt{ex}{thm}
\newtheorem{ex}[ex]{Example}
\aliascntresetthe{ex}
\crefname{ex}{Example}{Examples}

\theoremstyle{remark}
\newaliascnt{rmk}{thm}
\newtheorem{rmk}[rmk]{Remark}
\aliascntresetthe{rmk}
\crefname{rmk}{Remark}{Remarks}

\newaliascnt{clm}{thm}

\aliascntresetthe{clm}
\crefname{clm}{Claim}{Claims}

\numberwithin{equation}{section}

\newcommand{\nc}[3]{\newcommand{#2}[#1]{#3}}
\nc3{\rnc}{\renewcommand{#2}[#1]{#3}}

\nc1{\pr}{\left(#1\right)}
\nc1{\ang}{\langle#1\rangle}
\nc1{\ov}{\overline{#1}}
\nc1{\tld}{\widetilde{#1}}
\nc1{\card}{\left|#1\right|}

\nc1{\fk}{\mathfrak{#1}}

\nc0{\al}{\alpha}
\nc0{\be}{\beta}
\nc0{\ga}{\gamma}
\nc0{\de}{\delta}
\nc0{\s}{\sigma}
\nc0{\Om}{\Omega}

\nc0{\sub}{\subseteq}
\nc0{\iso}{\cong}

\nc1{\arw}{\overset {#1} \longrightarrow}
\nc0{\C}{\mathbb{C}}
\nc0{\F}{\mathbb{F}}
\nc0{\N}{\mathbb{N}}
\nc0{\R}{\mathbb{R}}
\nc0{\Q}{\mathbb{Q}}
\nc0{\Z}{\mathbb{Z}}

\nc0{\cH}{\mathcal{H}}
\nc0{\cO}{\mathcal{O}}

\nc0{\x}{\times}
\nc0{\divides}{\mid}
\nc0{\ndiv}{\nmid}
\nc0{\Op}{\bigoplus}

\nc0{\inv}{^{-1}}
\nc2{\cnj}{{}^{#1}#2}

\nc0{\invlim}{{\underset \longleftarrow \lim}}
\nc0{\dirlim}{{\underset \longrightarrow \lim}}

\DeclareMathOperator{\Br}{Br}
\DeclareMathOperator{\Gal}{Gal}
\DeclareMathOperator{\id}{id}
\DeclareMathOperator{\Nm}{Nm}
\DeclareMathOperator{\Aut}{Aut}

\nc0{\ab}{\text{ab}}
\DeclareMathOperator{\tame}{tame}
\nc1{\tm}{#1^{\tame}}
\DeclareMathOperator{\unr}{ur}
\nc1{\un}{#1^{\unr}}

\nc1{\G}{G_{#1}}

\nc0{\p}{\fk p}
\nc0{\q}{\fk q}

\nc1{\qell}{#1^\x / #1^{\x\ell}} 
\nc1{\perfect}{#1^\x = #1^{\x\ell}} 

\nc0{\fin}{\text{fin}}
\nc0{\Places}{\mathcal{P}}

\title{Pro-$\ell$-by-cyclotomic and tamely ramified variants of the Neukirch-Uchida Theorem}

\author{Ido Karshon}

\author{Mark Shusterman} \thanks{M.S. is The Dr. A. Edward Friedmann Career Development Chair in Mathematics}

\address{Faculty of Mathematics and Computer Science, Weizmann
Institute of Science, 234 Herzl Street, Rehovot 76100, Israel.}

\email{ido.karshon@weizmann.ac.il}

\address{Faculty of Mathematics and Computer Science, Weizmann
Institute of Science, 234 Herzl Street, Rehovot 76100, Israel.}

\email{mark.shusterman@weizmann.ac.il}

\date{\today}

\begin{document}

\begin{abstract}
    We prove a generalization of the Neukirch-Uchida Theorem. In particular,
    we show that the isomorphism type of a number field $K$
    can be recovered from the maximal pro-$\ell$-by-cyclotomic quotient of its absolute Galois group $\G{\ov{K}/K}$.
    This should be contrasted with the previous result that the isomorphism type cannot, in general, be recovered from the maximal pronilpotent quotient.
    We also show that the isomorphism type can be recovered from the maximal tamely ramified quotient.
\end{abstract}

\maketitle
\section{Introduction}\label{sec:introduction}

The Neukirch-Uchida Theorem states the following.
\begin{thm}\label{thm:classical neukirch uchida}
    Let $K_1, K_2$ be number fields with algebraic closures $\ov{K_1}, \ov{K_2}$.
    If $\al:\G{\ov{K_1}/K_1} \to \G{\ov{K_2}/K_2}$ is an isomorphism of profinite groups,
    then there exists an isomorphism of field extensions $\s:\ov{K_1}/K_1 \to \ov{K_2}/K_2$ that induces $\al$,
    and it is unique.
\end{thm}

Here, an isomorphism of field extensions $\s:L_1/K_1 \to L_2/K_2$
is an isomorphism $\s:L_1 \to L_2$ such that $\s(K_1) = K_2$.

There have been many results generalizing \cref{thm:classical neukirch uchida}
to the case where $\G{\ov{K_i}/K_i}$ is replaced by some quotient of it.
For instance, a Neukirch-Uchida variant was proven
for the maximal prosolvable quotient \cite{solvable Neukirch-Uchida},
and later for the maximal three-step solvable quotient \cite{3-solvable Neukirch-Uchida},
although the uniqueness requirement on $\s$ needs to be weakened in this case.
Variants with restricted ramification have been proven in
\cite{restricted ramification Neukirch-Uchida, 
restricted ramification Neukirch-Uchida 2,
restricted ramification Neukirch-Uchida with stable sets}.
A variant for $p$-closed extensions
(studying a property which we also consider in this paper,
but obtaining a result of a very different nature)
was proven in \cite{p-closed almost everywhere Neukirch-Uchida}.
Recently, a Neukirch-Uchida variant was proven for the quotient $\G{\un{K(\mu_\infty)}/K}$ \cite{unramified-by-cyclotomic Neukirch-Uchida}.
Further work on Neukirch-Uchida variants for number fields
includes \cite{pro-C Neukirch-Uchida} 
and a work in progress by Pop and Topaz.
As for negative results,
it was shown that the two-step nilpotent quotient of $\G{\ov K/K}$
does not recover $K$
\cite{2-nilpotent anti-Neukirch-Uchida},
and later that the pronilpotent quotient does not recover it \cite{nilpotent anti-Neukirch-Uchida}.
It was also shown that the isomorphism types of the $p$-Sylow subgroups of $\G{\ov K/K}$ 
do not recover $K$ \cite{sylow anti-Neukirch-Uchida}.

In this paper we prove another variant of the Neukirch-Uchida Theorem.
Our main technical result is stated in \cref{thm:neukirch uchida},
and its formulation uses \cref{def:ell-sealed} and \cref{def:abundance}.
Here we state some consequences.

\begin{thm}\label{thm:pro-ell-by-cyclotomic neukirch uchida}
    Let $K_1, K_2$ be number fields
    and let $\Om_i/K_i$ be the maximal pro-$\ell$ extension
    of $K_i(\mu_\infty)$.
    If $\al:\G{\Om_1/K_1} \to \G{\Om_2/K_2}$ is an isomorphism of profinite groups,
    then there exists an isomorphism of field extensions $\s:\Om_1/K_1 \to \Om_2/K_2$
    that induces $\al$,
    and it is unique.
\end{thm}

It also follows from \cref{thm:neukirch uchida} that this holds
when $K(\mu_\infty)$ is replaced by the maximal abelian extension of $K$.
Thus, the maximal pro-$\ell$-by-abelian quotient of $\G{\ov K/K}$ determines $K$.
Since a pronilpotent group is isomorphic to the direct product of its Sylow subgroups,
it follows that the maximal pronilpotent-by-abelian
quotient also determines $K$.
In \cite{nilpotent anti-Neukirch-Uchida},
as stated before,
it is shown that the maximal pronilpotent quotient of $\G{\ov K/K}$
does not determine $K$.

We note that the maximal prosupersolvable quotient lies strictly between the maximal pronilpotent-by-abelian
and the maximal pronilpotent quotients.
Further, the maximal pro-$\ell$ extension of $K(\mu_\ell)$
(which is relevant to our proof)
is a prosupersolvable extension of $K$, being an extension of a pro-$\ell$ group by an abelian group of exponent dividing $\ell - 1$. 
These observations suggest
the prosupersolvable Neukirch-Uchida variant 
as a direction for future research.

\begin{thm}\label{thm:tamely ramified neukirch uchida}
    Let $K_1, K_2$ be number fields
    and let $\tm{K_i}$ denote the maximal tamely ramified extension of $K_i$.
    If $\al:\G{\tm{K_1}/K_1} \to \G{\tm{K_2}/K_2}$ is an isomorphism of profinite groups,
    then there exists an isomorphism of field extensions $\s:\tm{K_1}/K_1 \to \tm{K_2}/K_2$
    that induces $\al$, and it is unique.
\end{thm}
This can be seen as another restricted ramification Neukirch-Uchida variant,
this time eliminating wild ramification at all primes
instead of eliminating ramification at a subset of primes.

The structure of the paper is as follows.
\begin{itemize}
    \item In \cref{sec:notation}
    we define our notation.
    \item In \cref{sec:ell-sealed-cohomology}
    we define the notion of $\ell$-sealed fields
    and study their cohomologies.
    \item In \cref{sec:brauer-sequence}
    we prove an $\ell$-sealed analogue of the Brauer exact sequence.
    \item In \cref{sec:neukirch}
    we use the Brauer sequence to prove
    an $\ell$-sealed analogue of Neukirch's Theorem.
    The section begins with an overview of this proof.
    \item In \cref{sec:char-detect}
    we define abundant fields,
    which allow us to distinguish between primes of different residue characteristics
    using only group-theoretic information.
    We get a strengthening of Neukirch's Theorem
    in the case of extensions that are both $\ell$-sealed and abundant.
    \item In \cref{sec:prime-bijection}
    we construct a bijection between the primes of the base fields
    that preserves residue characteristics.
    We deduce that the base fields are arithmetically equivalent.
    \item In \cref{sec:uchida-isomorphism}
    we adapt Uchida's method from \cite{solvable Neukirch-Uchida}
    to lift the arithmetic equivalence
    into an isomorphism of the base fields,
    under the assumption that the base fields contain $\mu_\ell$.
    This can already be seen as a weak version of Neukirch-Uchida.
    Originally Uchida used this method to prove the solvable Neukirch-Uchida variant
    in full generality,
    but we choose to give a simplified version of his argument 
    that only handles the case with $\mu_\ell$ and only proves an isomorphism of the base fields.
    \item In \cref{sec:full-neukirch-uchida}
    we deduce the full Neukirch-Uchida Theorem
    from the previous weak version, finishing the proof of our main result.
    The section begins with an overview of the proof.
    We also prove \cref{thm:pro-ell-by-cyclotomic neukirch uchida} and \cref{thm:tamely ramified neukirch uchida}.
\end{itemize}

\section{Notation and Preliminaries}\label{sec:notation}

We write $\ov K$ for the algebraic closure of $K$.
If $K$ is a number field or a local field, we also write
$\un K$ for its maximal unramified extension
and $\tm K$ for its maximal tamely ramified extension.
We write $\G{L/K}$ for $\Gal(L/K)$
and $\G{K}$ for $\G{\ov K/K}$.
When we refer to a number field or a local field, 
we always assume it has finite degree over its prime field.

Let $E$ be an algebraic extension of $\Q$.
We denote the set of places of $E$ by $\Places(E)$;
if $E/\Q$ has infinite degree, this is defined as $\invlim_{K \sub E} \Places(K)$,
where $K$ runs over the number fields contained in $E$,
and the limit is taken with respect to restrictions.
We refer to the elements of $\Places(\Q)$
as rational places.
Note that $\Places(E)$ has the structure of a locally profinite topological space.
We write $\Places_\fin(E)$ for the subspace of non-Archimedean places.
Given $S \sub \Places(E)$ and an algebraic extension $E'/E$,
we write $S(E')$ for the set of places of $E'$ lying over a place in $S$.
If $S = \{\p\}$ is a singleton then we also write $S(E')$ as $\Places_\p(E')$.
We say that a place $\p \in \Places(E)$
is decomposable in $E'/E$
if $\card{\Places_\p(E')} > 1$, and indecomposable otherwise.
We also define $E_\p = \bigcup_{K \sub E} K_{\p|_K}$;
note that this field might not be complete.
For a Galois extension $E'/E$ and $\fk{P} \in \Places(E')$,
we write $\G{E'/E, \fk{P}}$ for the subgroup of $\G{E'/E}$ that fixes $\fk{P}$,
which is isomorphic to $\G{E'_\fk{P}/E_{\fk{P}|_E}}$.

Let $F$ be a non-Archimedean local field. 
We write $\cO_F$ for its ring of integers
and $\fk m_F$ for its maximal ideal.
We denote the inertia and ramification degrees of
a finite extension $E/F$ 
by $f_{E/F}$ and $e_{E/F}$ respectively.

We denote the group of $n$th roots of unity in the algebraic closure of a field of characteristic zero by $\mu_n$,
and the union of $\mu_n$ for all $n \ge 1$ by $\mu_\infty$.
For a prime $p$, we denote the union of $\mu_{p^n}$ for all $n \ge 1$ by $\mu_{p^\infty}$.

We denote the cardinality of a finite set $A$ by $\card{A}$,
and we use the same notation for the supernatural cardinalities of profinite sets.
Whenever we refer to a subgroup 
of a profinite group we implicitly require it to be closed,
and we similarly require homomorphisms between profinite groups
and modules over profinite groups to be continuous.

For elements $g, \s$ of a group, we denote $\cnj\s{g} = \s g \s\inv$.
We extend this notation to the conjugation of subgroups and conjugation of modules.
If $A$ is an abelian group we write $A[n]$ for its $n$-torsion subgroup.

It is often the case that we have a number field $K$, a Galois extension $\Om/K$,
and a place $\fk{P}$ of $\Om$ lying over a place $\p$ of $K$,
and we consider a cohomology group of $\G{\Om_\fk{P}/K_\p}$ such as
$H^i(\G{\Om_\fk{P}/K_\p}, \Om_\fk{P}^\x)$.
Suppose $\fk{P}'$ is another prime of $\Om$ lying over $\p$.
Then there is $g \in \G{\Om/K}$ such that $g\fk{P} = \fk{P}'$,
and we get a natural sequence of isomorphisms
\[
    H^i(\G{\Om_\fk{P}/K_\p}, \Om_\fk{P}^\x) \iso H^i(\G{\Om/K, \fk{P}}, \Om_\fk{P}^\x)
    \iso H^i(\cnj{g}{\G{\Om/K, \fk{P}}}, \cnj{g}{\Om_\fk{P}^\x})
    \iso H^i(\G{\Om/K, g\fk{P}}, \Om_{g\fk{P}}^\x)
    \iso H^i(\G{\Om_{\fk{P}'}/K_\p}, \Om_{\fk{P}'}^\x).
\]
In fact, the standard theory of group cohomology
shows that the composition does not depend
on the choice of $g$.
Denoting it by 
$\Phi_{\fk{P}\fk{P}'}$,
it is also not difficult to show that $\Phi_{\fk{P}'\fk{P}''} \circ \Phi_{\fk{P}\fk{P}'} = \Phi_{\fk{P}\fk{P}''}$.
We can therefore define the "cohomology group"
$H^i(\G{\Om^\p/K_\p}, \Om^{\p\x})$
by $\pr{\Op_{\fk{P} \in \Places_\p(\Om)}H^i(\G{\Om_\fk{P}/K_\p},\Om_\fk{P}^\x)}/R$
where $R$ is generated by the relations of the form $[\al]_\fk{P} - [\Phi_{\fk{P}\fk{P}'}(\al)]_{\fk{P}'}$.
This is a canonical construction depending only on $\p$,
with natural isomorphisms to each $H^i(\G{\Om_\fk{P}/K_\p}, \Om_\fk{P}^\x)$
that are compatible with $\Phi_{\fk{P}\fk{P}'}$.
Maps involving this notation
are implicitly required to be natural with respect to the isomorphisms
$\Phi_{\fk{P}\fk{P}'}$.
For instance, there is a well-defined restriction map
$\Res : H^i(\G{\Om/K}, \Om^\x) \to H^i(\G{\Om^\p/K_\p}, \Om^{\p\x})$
because the following diagram commutes:
\[
    \xymatrix{
        H^i(\G{\Om/K}, \Om^\x) \ar[d]^\Res\ar[rd]^\Res \\
        H^i(\G{\Om_\fk{P}/K_\p}, \Om_\fk{P}^\x) \ar[r]^{\Phi_{\fk{P}\fk{P}'}} & H^i(\G{\Om_{\fk{P}'}/K_\p}, \Om_{\fk{P}'}^\x).
    }
\]
Other "cohomology groups" of $\G{\Om^\p/K_\p}$ include
$H^i(\G{\Om^\p/K_\p}, \qell{{\Om^\p}})$,
$H^i(\G{\Om^\p/K_\p}, \mu_\ell)$ when $\mu_\ell \sub \Om$,
etc.
\section{Cohomologies of \texorpdfstring{$\ell$}{ell}-Sealed Fields}\label{sec:ell-sealed-cohomology}

From now on,
$\ell$ denotes a prime number
and $S$ denotes a finite set of rational places.

\begin{defn}\label{def:ell-sealed}
    Let $\ell$ be a prime number and let $\Om$ be an algebraic extension of $\Q$.
    \begin{enumerate}
        \item We say $\Om$ is $\ell$-sealed with respect to $S$
        if $\mu_\ell \sub \Om$
        and there are no $\Z/\ell\Z$-extensions of $\Om$
        where every place of $S(\Om)$ splits completely.
        \item We say $\Om$ is $\ell$-sealed if it is $\ell$-sealed with respect to some $S$.
    \end{enumerate}
\end{defn}

For a number field $K$,
consider the map $\Phi_{K,\ell,S}:\qell{K} \to \Op_{\p \in S(K)} \qell{K_\p}$.

\begin{prop}\label{prop:K^x/K^ell local surjection}
    $\Phi_{K,\ell, S}$ is surjective.
\end{prop}
\begin{proof}
    Let $(\ov{\al_\fk{P}})_{\fk{P} \in S(K)} \in \Op_{\fk{P} \in S(K)} \qell{K_\fk{P}}$,
    with $\al_\fk{P} \in K_\fk{P}^\x$
    representatives of the classes $\ov{\al_\fk{P}}$.
    By \cite[Chap. II, 5.7]{algebraic number theory Neukirch}, 
    it follows that the subsets $K_\fk{P}^{\x\ell} \sub K_\fk{P}^\x$
    are open.
    Thus, by the Chinese Remainder Theorem,
    there is $\al \in K$
    such that the quotient $\frac{\al}{\al_\fk{P}}$
    belongs to $K_\fk{P}^{\x\ell}$
    for all $\fk{P} \in S(K)$.
    This implies $\Phi_{K,\ell, S}(\al K^{\x\ell}) = (\ov{\al_\fk{P}})_{\fk{P} \in S(K)}$.
\end{proof}

We can extend the previous definition of $\Phi_{K,\ell,S}$
to general algebraic extensions $\Om/\Q$,
defining $\Phi_{\Om,\ell,S}$
as the composition 
\[\qell{\Om} \to \dirlim_{K \sub \Om}\qell{K} \to \dirlim_{K \sub \Om}\Op_{\p \in S(K)}\qell{K_\p}\]
where the direct limits are over number fields $K$ contained in $\Om$.
From \cref{prop:K^x/K^ell local surjection}
it is immediate that $\Phi_{\Om,\ell,S}$ is surjective.

\begin{prop}\label{prop:ell-sealed alt def}
    $\Om$ is $\ell$-sealed with respect to $S$
    if and only if $\mu_\ell \sub \Om$ and $\Phi_{\Om,\ell, S}$ is an injection (and thus an isomorphism).
\end{prop}

\begin{proof}
    Suppose that $\mu_\ell \sub \Om$ and $\Phi_{\Om,\ell, S}$ is an injection,
    and assume for the sake of contradiction
    that $\Om$ has a $\Z/\ell\Z$-extension 
    in which every place of $S(\Om)$ splits completely.
    This extension has the form $\Om(\al^\frac1\ell)/\Om$
    for some $\al \in \Om^\x \setminus \Om^{\x\ell}$.
    It follows that 
    for every $\fk{P} \in S(\Om)$,
    $\al$ is an $\ell$th power in $\Om_\fk{P}$,
    meaning that there is a number field $L \sub \Om$
    such that $\al$ is an $\ell$th power in $L_{\fk{P}|_L}$.
    Thus the sets $U_L = \{ \fk{P} \in S(\Om) \mid \al \text{ is an } \ell\text{th power in } L_{\fk{P}|_L} \}$
    form an open cover of the compact space $S(\Om)$.
    Since $U_{L_1} \cup U_{L_2} \sub U_{L_1L_2}$
    it follows there is a number field $L \sub \Om$
    such that $\al$ is an $\ell$th power in $L_{\fk{P}|_L}$
    for every $\fk{P} \in S(\Om)$.
    However, this implies $\al\Om^{\x\ell} \in \ker(\Phi_{\Om, \ell, S})$,
    so that $\al \in \Om^{\x\ell}$, a contradiction.
    
    Conversely, suppose $\Om$ is $\ell$-sealed with respect to $S$.
    Let $\al \in \Om^\x$ and assume $\al\Om^{\x\ell} \in \ker(\Phi_{\Om, \ell, S})$.
    Then there is a number field $L \sub \Om$
    containing $\mu_\ell$ and $\al$, such that $\al$ is an $\ell$th power in $L_{\p}$ for every $\p \in S(L)$.
    It follows that every place of $S(L)$ splits completely in $L(\al^\frac1\ell)/L$,
    and thus every place of $S(\Om)$ splits completely in $\Om(\al^\frac1\ell)/\Om$.
    This implies $\Om(\al^\frac1\ell) = \Om$ so that $\al \in \Om^{\x\ell}$.
\end{proof}

\begin{ex}\label{ex:ell-sealed-examples}
    Let $\Om$ be an algebraic extension of $\Q$ and $\ell$ a prime number.
    Then:
    \begin{enumerate}
        \item If $\Om$ contains $\mu_\ell$ and has no $\Z/\ell\Z$-extensions,
        then $\Om$ is $\ell$-sealed with respect to $S = \emptyset$.
        \item If $\Om$ has no nontrivial tamely ramified extensions,
        then $\Om$ is $\ell$-sealed with respect to $S = \{\ell\}$.
    \end{enumerate}
    In particular, for a number field $K$,
    both $\tm K$ and the maximal pro-$\ell$ extension of $K(\mu_\ell)$ are $\ell$-sealed.
\end{ex}

\begin{proof}
    The first part is trivial.
    For the second part,
    suppose $\al \in \Om^\x$ satisfies $\al\Om^{\x\ell} \in \ker(\Phi_{\Om, \ell, \{\ell\}})$.
    Then there is a number field $K \sub \Om$
    such that $\al \in K^\x$, $\mu_\ell \sub K$, and $\al \in K_\p^{\x\ell}$ 
    for all $\p \in \Places_\ell(K)$.
    Let $L = K(\al^\frac1\ell)$,
    and note that every place of $\Places_\ell(K)$
    splits completely in $L$.
    Since $L/K$ is a Galois extension of degree $1$ or $\ell$,
    it can have no wild ramification at places not lying over $\ell$,
    so it is tamely ramified.
    Tame ramification is preserved under base change,
    so $\Om(\al^\frac1\ell)/\Om$ is also tamely ramified,
    implying that $\al \in \Om^{\x\ell}$
    by the assumption on $\Om$.
    Thus, $\ker(\Phi_{\Om, \ell, \{\ell\}})$ is trivial.
\end{proof}

\begin{lem}\label{lem:O/ell=0 outside S}
    Suppose $\Om$ is $\ell$-sealed with respect to $S$.
    Then for every $\fk{P} \in \Places(\Om) \setminus S(\Om)$ we have $\perfect{\Om_\fk{P}}$.
\end{lem}

\begin{proof}
    Let $\al_\fk{P} \in \Om_\fk{P}^\x$.
    Let $K \sub \Om$ be a number field containing $\mu_\ell$
    such that $\al_\fk{P} \in K_{\fk{P}|_K}$.
    By \cref{prop:K^x/K^ell local surjection}
    for the set $S \cup \{\fk{P}|_\Q\}$,
    there is $\al \in K^\x$
    whose images in $K_\p, \p \in S(K)$
    are $\ell$th powers,
    while its image in $K_{\fk{P}|_K}$
    lies in $\al_\fk{P} \cdot K_{\fk{P}|_K}^{\x\ell}$.
    It follows by \cref{prop:ell-sealed alt def} 
    that $\al \in \Om^{\x\ell}$ and therefore
    $\al_\fk{P}\in \Om_\fk{P}^{\x\ell}$.
\end{proof}

\begin{lem}\label{lem:Xi/ell=0 implies ell^infty | [Xi:F]}
    Let $F$ be a non-Archimedean local field of characteristic zero,
    and let $\Xi$ be an algebraic extension of $F$
    containing $\mu_\ell$
    and satisfying $\perfect{\Xi}$.
    Then $\ell^\infty \divides [\Xi:F]$.
\end{lem}

\begin{proof}
    Assume, for the sake of contradiction,
    that $\ell^\infty \ndiv [\Xi:F]$.
    Then there exists a finite subextension $F'/F$ of $\Xi/F$
    such that $\ell \ndiv [\Xi:F']$.
    We may assume without loss of generality that $\mu_\ell \sub F'$.
    Let $n$ be an integer
    such that the map $\mu_\ell \to \cO_{F'}/\fk m_{F'}^n\cO_{F'}$
    is injective.

    The $\ell$-power map $(\cO_{F'}/\fk m_{F'}^n\cO_{F'})^\x \arw\ell (\cO_{F'}/\fk m_{F'}^n\cO_{F'})^\x$
    is a homomorphism between finite groups of the same size,
    and has kernel of size $\ell$.
    In particular it is not surjective,
    and there is $\al \in \cO_{F'}$
    whose image in $\cO_{F'}/\fk m_{F'}^n\cO_{F'}$
    is not an $\ell$th power.
    Therefore $\al$ is not an $\ell$th power in $F'$,
    and $F'(\al^\frac1\ell)/F'$ is a $\Z/\ell\Z$-extension.
    However, from $\perfect{\Xi}$ we find $\al^\frac1\ell \in \Xi$.
    This implies $\ell \divides [\Xi : F']$,
    contradicting the assumption.
\end{proof}

\begin{lem}\label{lem:O/ell cohomology local decomposition}
    Let $K$ be a number field and let $\Om/K$ be a Galois extension which is $\ell$-sealed with respect to $S$.
    Then for every $i \ge 0$ the restriction maps induce an isomorphism
    \[
        H^i\pr{\G{\Om/K}, \qell{\Om}} \to \Op_{\p \in S(K)}H^i\pr{\G{\Om^\p/K_\p}, \qell{{\Om^\p}}}.
    \]
\end{lem}

\begin{proof}
    \begin{align*}
        &H^i(\G{\Om/K}, \qell{\Om}) \\
        &\iso \dirlim_{L \sub \Om\text{, }L/K \text{ Galois}} H^i\pr{\G{L/K}, \Op_{\fk{P} \in S(L)}\qell{L_\fk{P}}}\\
        &\iso \Op_{\p \in S(K)} \dirlim_{L \sub \Om\text{, }L/K \text{ Galois}} H^i\pr{\G{L/K}, \Op_{\fk{P} \in \Places_\p(L)}\qell{L_\fk{P}}}\\
        &\iso \Op_{\p \in S(K)} \dirlim_{L \sub \Om\text{, }L/K \text{ Galois}} H^i\pr{\G{L^\p/K_\p}, \qell{{L^\p}}}\\
        &\iso \Op_{\p \in S(K)} H^i\pr{\G{\Om^\p/K_\p}, \qell{{\Om^\p}}}\\
    \end{align*}
    where the first isomorphism follows from \cite[1.5.1]{cohomology of number fields}
    with inverse systems $\G{\Om/K} \iso \invlim_{L \sub \Om\text{, }L/K\text{ Galois}} \G{L/K}$
    and $\qell{\Om} \iso \invlim_{L\sub \Om\text{, }L/K\text{ Galois}}\Op_{\fk{P} \in S(L)}\qell{L_\fk{P}}$,
    the second isomorphism is obvious,
    the third isomorphism follows from a $\G{L^\p/K_\p}$-version of Shapiro's lemma,
    and the fourth isomorphism follows from a $\G{L^\p/K_\p}$-version of \cite[1.5.1]{cohomology of number fields}.
\end{proof}

\section{Brauer Exact Sequence for \texorpdfstring{$\ell$}{ell}-sealed extensions}\label{sec:brauer-sequence}

In this section we prove an analogue of the
Brauer exact sequence
for the case of $\ell$-sealed extensions.
Namely, we prove that the sequence
$0 \to H^2(\G{\Om/K}, \mu_\ell) \to \Op_{\p \in \Places(K)} H^2(\G{\Om^\p/K_\p}, \mu_\ell) \to \Z/\ell\Z$
is exact, where $\Om$ is an $\ell$-sealed Galois extension of $K$.

\begin{lem}\label{lem:five-term-exact-sequence}
    Let $K$ be a field of characteristic zero,
    and let $\Om/K$ be a Galois extension containing $\mu_\ell$.
    Then there is a five-term exact sequence
    \[
        \qell{K} \to H^0(\G{\Om/K}, \qell{\Om}) \to H^2(\G{\Om/K}, \mu_\ell) \to \Br(\Om/K)[\ell] \to H^1(\G{\Om/K}, \qell{\Om})
    \]
    Further, if $K$ is an algebraic extension of $\Q$ and $\p \in \Places(K)$,
    then this sequence is natural with respect to restrictions,
    in the sense that the restriction maps induce a commutative diagram
    \[
        \xymatrix@C-1pc{
            \qell{K} \ar[r]\ar[d]& H^0(\G{\Om/K}, \qell{\Om}) \ar[r]\ar[d]& H^2(\G{\Om/K}, \mu_\ell) \ar[r]\ar[d]& \Br(\Om/K)[\ell] \ar[r]\ar[d]& H^1(\G{\Om/K}, \qell{\Om})\ar[d]\\
            \qell{K_\p} \ar[r]& H^0(\G{\Om^\p/K_\p}, \qell{{\Om^\p}}) \ar[r]& H^2(\G{\Om^\p/K_\p}, \mu_\ell) \ar[r]& \Br(\Om^\p/K_\p)[\ell] \ar[r]& H^1(\G{\Om^\p/K_\p}, \qell{{\Om^\p}}).
        }
    \]
\end{lem}

\begin{proof}
    One way to show this is to consider the exact sequence $0 \to \mu_\ell \to \Om^\x \arw\ell \Om^\x \to \qell{\Om} \to 0$,
    which induces a spectral sequence converging to zero
    whose first page is as follows:
    \[
        \xymatrix{
            \dots & \dots & \dots & \dots \\
            H^2(\G{\Om/K},\mu_\ell) \ar[r] & \Br(\Om/K) \ar[r]^\ell & \Br(\Om/K) \ar[r] & H^2(\G{\Om/K}, \qell{\Om}) \\
            H^1(\G{\Om/K},\mu_\ell) \ar[r] & 0 \ar[r] & 0 \ar[r] & H^1(\G{\Om/K}, \qell{\Om}) \\
            H^0(\G{\Om/K},\mu_\ell) \ar[r] & K^\x \ar[r]^\ell & K^\x \ar[r] & H^0(\G{\Om/K}, \qell{\Om}). \\
        }
    \]
    For simplicity, we provide the elementary version of this proof.
    The short exact sequence
    $0 \to \mu_\ell \to \Om^\x \arw\ell \Om^{\x\ell} \to 0$
    induces the exact cohomology sequence
    \begin{equation}
        0 \to H^1(\G{\Om/K}, \Om^{\x\ell}) \arw\de H^2(\G{\Om/K}, \mu_\ell) \to \Br(\Om/K) \to H^2(\G{\Om/K}, \Om^{\x\ell}) \label{eq:1}
    \end{equation}
    where we used Hilbert's Theorem 90 for $H^1(\G{\Om/K}, \Om^\x) = 0$.
    Likewise, the short exact sequence
    $0 \to \Om^{\x\ell} \to \Om^\x \to \qell{\Om} \to 0$
    induces the exact sequences
    \begin{equation}
        0 \to \Om^{\x\ell} \cap K^\x \to K^\x \to H^0(\G{\Om/K}, \qell{\Om}) \arw\de H^1(\G{\Om/K}, \Om^{\x\ell}) \to 0 \label{eq:2}
    \end{equation} 
    and 
    \begin{equation}
        0 \to H^1(\G{\Om/K}, \qell{\Om}) \arw\de H^2(\G{\Om/K}, \Om^{\x\ell}) \to \Br(\Om/K) \to H^2(\G{\Om/K}, \qell{\Om}). \label{eq:3}
    \end{equation}
    We then define a map $H^0(\G{\Om/K}, \qell{\Om}) \to H^2(\G{\Om/K}, \mu_\ell)$
    as the composition
    \[
        H^0(\G{\Om/K}, \qell{\Om}) \arw\de H^1(\G{\Om/K}, \Om^{\x\ell}) \arw\de H^2(\G{\Om/K}, \mu_\ell)
    \]
    where the first map is from \eqref{eq:2} and the second is from \eqref{eq:1}.
    We also define a map $\Br(\Om/K)[\ell] \to H^1(\G{\Om/K}, \qell{\Om})$,
    by mapping an element $\al \in \Br(\Om/K)[\ell]$ to its image $\be \in H^2(\G{\Om/K}, \Om^{\x\ell})$ via \eqref{eq:1},
    noting that a subsequent mapping of $\be$ into $\Br(\Om/K)$ via \eqref{eq:3} 
    would equal $\ell\al = 0$
    (because the composition $\Om^\x \arw\ell \Om^{\x\ell} \to \Om^\x$
    is just multiplication by $\ell$),
    and concluding that $\be$ has a unique preimage in $H^1(\G{\Om/K}, \qell{\Om})$ by exactness of the sequence \eqref{eq:3}.
    It is then straightforward to verify that these maps
    fit into the desired exact sequence.

    The second part follows from the naturality of the above construction.

\end{proof}

\begin{cor}\label{cor:H^2(mu_ell) equals Brauer in local non-ell case}
    Let $F$ be a field of characteristic zero
    and let $\Xi/F$ be a Galois extension
    such that $\mu_\ell \sub \Xi$ and $\perfect{\Xi}$.
    Then the natural map $H^2(\G{\Xi/F}, \mu_\ell) \to \Br(\Xi/F)[\ell]$
    is an isomorphism.
\end{cor}

\begin{proof}
    Since $\qell{\Xi}=0$,
    the five-term exact sequence from \cref{lem:five-term-exact-sequence}
    degenerates into an isomorphism 
    $H^2(\G{\Xi/F}, \mu_\ell) \to \Br(\Xi/F)[\ell]$.
\end{proof}

\begin{thm}\label{thm:brauer sequence}
    Let $K$ be a number field and let $\Om/K$ be an $\ell$-sealed Galois extension.
    Then the restriction maps induce an exact sequence
    \[
        0 \to H^2(\G{\Om/K}, \mu_\ell) \to \Op_{\p \in \Places(K)}H^2(\G{\Om^\p/K_\p}, \mu_\ell) \to \Z/\ell\Z
    \]
    which fits into a commutative diagram
    \[
        \xymatrix{
            0 \ar[r]& H^2(\G{\Om/K}, \mu_\ell) \ar[r]\ar[d]& \Op_{\p \in \Places(K)}H^2(\G{\Om^\p/K_\p}, \mu_\ell) \ar[r]\ar[d]& \Z/\ell\Z\ar@{=}[d]\\
            0 \ar[r] & \Br(\Om/K)[\ell] \ar[r]& \Op_{\p \in \Places(K)}\Br(\Om^\p/K_\p)[\ell] \ar[r]& \Z/\ell\Z
        }
    \]
    with exact rows. 
\end{thm}

\begin{proof}
    Consider the commutative diagram
    \[
        \xymatrix@C+3pc{
            \qell{K} \ar[r]\ar[d]& \Op_{\p \in S(K)}\qell{K_\p} \ar[d]\\
            H^0(\G{\Om/K}, \qell{\Om}) \ar[r]\ar[d]& \Op_{\p \in S(K)}H^0(\G{\Om^\p/K_\p}, \qell{{\Om^\p}})\ar[d]\\
             H^2(\G{\Om/K}, \mu_\ell) \ar[r]\ar[d]& \Op_{\p \in \Places(K)}H^2(\G{\Om^\p/K_\p}, \mu_\ell)\ar[d]\\
             \Br(\Om/K)[\ell] \ar[r]\ar[d]& \Op_{\p \in \Places(K)}\Br(\Om^\p/K_\p)[\ell]\ar[d]\\
             H^1(\G{\Om/K}, \qell{\Om})\ar[r]& \Op_{\p \in S(K)}H^1(\G{\Om^\p/K_\p}, \qell{{\Om^\p}}).\\
        }
    \]

    In the second column, note that the third and fourth terms are summed over $\Places(K)$
    while the first, second and fifth terms are summed over $S(K)$.
    Thus, by \cref{lem:five-term-exact-sequence} and \cref{cor:H^2(mu_ell) equals Brauer in local non-ell case},
    the diagram has exact columns.
    By \cref{prop:K^x/K^ell local surjection} 
    the top horizontal map is surjective
    and by \cref{lem:O/ell cohomology local decomposition} the second and fifth horizontal maps are isomorphisms.
    We also know, by the classical Brauer exact sequence, that the fourth horizontal map is injective, with cokernel isomorphic to $\Z/\ell\Z$ or to $0$.
    The required exact sequence now follows from a diagram chase.
\end{proof}

\section{Neukirch's Theorem}\label{sec:neukirch}

In this section we prove Neukirch's Theorem for $\ell$-sealed extensions
(\cref{thm:neukirch}).
Namely, for an $\ell$-sealed Galois extension $\Om/K$
and a subgroup $H \le \G{\Om/K}$,
if $H$ is isomorphic to a local Galois group satisfying certain conditions,
then $H$ is local in $\G{\Om/K}$;
that is, $H$ is contained in $\G{\Om/K, \fk{P}}$ for some place $\fk{P}$ of $\Om$.

Let us start with an outline of the proof.
We take $E = \Om^H$,
and study the analogue of the Brauer exact sequence
from \cref{sec:brauer-sequence}
for the extension $\Om/E$.
Since the global cohomology group $H^2(\G{\Om/E}, \mu_\ell)$
depends only on the isomorphism type of $H$ 
(under the assumption $\mu_\ell \sub K$),
the exact sequence provides us with information about the local cohomology groups
$H^2(\G{\Om^\fk{P}/E_\fk{P}}, \mu_\ell)$,
from which we deduce there is a unique place $\fk{P}_E$ of $E$
producing a nonzero local cohomology group.
The assumptions of the theorem hold for any open subgroup $H' \le H$
and its corresponding fixed field $E' = \Om^{H'}$,
so for any finite subextension $E'/E$
there is a unique place $\fk{P}_{E'}$ of $E'$
producing a nonzero local cohomology group.
We can use this system of places to show there is a place of $E$
which is indecomposable in $\Om$,
and thus $H \le \G{\Om/K, \fk{P}}$ for some place $\fk{P}$ of $\Om$.
This final part is easy in the case $\fk{P}_E \notin S(E)$,
where $S$ is a finite set of rational places with respect to which $\Om$ is $\ell$-sealed,
and more complicated in the case $\fk{P}_E \in S(E)$.
We handle the case $\fk{P}_E \in S(E)$ with a topological argument on the space $S(E)$,
utilizing its compactness which follows from the finiteness of $S$.

We now prove a sequence of lemmas and propositions
needed for the proof of Neukirch's Theorem.

\begin{prop}\label{prop:local groups ell-disjoint}
    Let $K$ be a number field with $\Om/K$ an $\ell$-sealed Galois extension.
    Then for any two distinct places $\fk{P}_1, \fk{P}_2 \in \Places(\Om)$
    we have $\ell \ndiv \card{\G{\Om/K, \fk{P}_1} \cap \G{\Om/K, \fk{P}_2}}$.
\end{prop}

\begin{proof}
    Let $S$ be a finite set of rational places
    such that $\Om$ is $\ell$-sealed with respect to it.
    Let $H = \G{\Om/K, \fk{P}_1} \cap \G{\Om/K, \fk{P}_2}$
    and $E = \Om^H$.
    Note that $\fk{P}_1|_E$ and $\fk{P}_2|_E$ are indecomposable in $\Om$,
    and thus have to be distinct.

    Assume, for the sake of contradiction, that $\ell \divides \card{H}$.
    Then there is a finite subextension $E''/E$ of $\Om/E$ with order divisible by $\ell$.
    We may assume $E''/E$ is Galois and $\mu_\ell \sub E''$.
    As $\G{E''/E}$ is a finite group of order divisible by $\ell$,
    it has an element of order $\ell$,
    and therefore there is a subfield $E''/E'/E$
    such that $[E'':E'] = \ell$.
    It follows that $E'$ contains $\mu_\ell$ as well,
    and by Kummer Theory there is $\al \in E'$ such that $E'' = E'(\al^\frac1\ell)$.
    
    Let $L \sub E'$ be a number field such that $\fk{P}_1|_L \ne \fk{P}_2|_L$, $\mu_\ell \sub L$, and $\al\in L$.
    By \cref{prop:K^x/K^ell local surjection}
    for the set $S \cup \{\fk{P}_1|_\Q, \fk{P}_2|_\Q\}$,
    there is $\be \in L^\x$
    such that the images of $\be$ in $L_\p, \p \in S(L) \cup \{\fk{P}_1|_L\}$
    are all $\ell$th powers,
    while its image in $L_{\fk{P}_2|_L}$
    lies in $\al \cdot L_{\fk{P}_2|_L}^{\x\ell}$. 
    It follows that $\be^\frac1\ell \in \Om$.
    Let $\tld{L} = L(\be^\frac1\ell) \sub \Om$.
    Since $\fk{P}_1|_{L}$ splits completely in $\tld{L}$,
    we find that $\fk{P}_1|_{E'}$ splits completely in $\tld{L}E'$.
    However, there is a unique place of $\Om$ lying over $\fk{P}_1|_{E'}$,
    so we must have $\tld{L} \sub E'$.
    Thus $\al^\frac1\ell \in \tld{L}_{\fk{P}_2|_{\tld{L}}} \sub E'_{\fk{P}_2|_{E'}}$,
    implying that $\fk{P}_2|_{E'}$ splits completely in $E'' = E'(\al^\frac1\ell)$.
    Since there is a unique place of $\Om$ lying over $\fk{P}_2|_{E'}$,
    this implies $E' = E''$,
    contradicting $[E'':E'] = \ell$.
\end{proof}

\begin{lem}\label{lem:subgroup cohomology}
    Let $G$ be a profinite group and $M$ be a $G$-module. Let $H \le G$. Then
    \[
        H^n(H, M) \iso \dirlim_{H \sub U \sub G} H^n(U, M)
    \]
    where $U$ runs through the open subgroups in $G$ containing $H$, and the limit is taken with respect to the restriction maps.
\end{lem}
\begin{proof}
    Apply \cite[1.5.1]{cohomology of number fields}
    to the inverse system $H = \invlim_{H \sub U \sub G} U$.
\end{proof}

\begin{lem}\label{lem:H2 of infinite local field}
    Let $\ell$ be a prime number and let $F$ be a non-Archimedean local field of characteristic zero.
    Let $\Xi/E/F$ be a tower of algebraic extensions
    with $\mu_\ell \sub \Xi$ and $\perfect{\Xi}$.
    Then $H^2(\G{\Xi/E}, \mu_\ell)$ is isomorphic to $\Z/\ell\Z$ if $\ell^\infty \ndiv [E:F]$
    and to $0$ otherwise.
\end{lem}

\begin{proof}
    By \cref{cor:H^2(mu_ell) equals Brauer in local non-ell case}
    we have $H^2(\G{\Xi/E}, \mu_\ell) \iso \Br(\Xi/E)[\ell]$.
    By \cref{lem:subgroup cohomology}
    we have $\Br(\Xi/E)[\ell] \iso \dirlim_{F'}\Br(\Xi/F')[\ell]$
    where $F'$ runs through the finite subextensions of $E/F$
    and the direct limit is taken with respect to the restriction maps.

    By \cref{lem:Xi/ell=0 implies ell^infty | [Xi:F]}
    we find that $\ell \divides [\Xi : F']$ for every such $F'$,
    so we have $\Br(\Xi/F')[\ell] \iso \Z/\ell\Z$.
    Further, for a tower $E/F''/F'/F$,
    the restriction map $\Br(\Xi/F')[\ell] \to \Br(\Xi/F'')[\ell]$
    is zero if $\ell \divides [F'' : F']$,
    and is an isomorphism otherwise.
    The result follows.
\end{proof}

\begin{cor}\label{cor:H2 unchanged by finite subextension}
    Let $\ell$ be a prime number and let $F$ be a non-Archimedean local field of characteristic zero.
    Let $\Xi/E'/E/F$ be a tower of algebraic extensions
    with $\mu_\ell \sub \Xi$, $\perfect{\Xi}$ and $E'/E$ finite.
    Then $H^2(\G{\Xi/E}, \mu_\ell) \ne 0$ if and only if $H^2(\G{\Xi/E'}, \mu_\ell) \ne 0$,
    and in this case they are both isomorphic to $\Z/\ell\Z$.
\end{cor}

\begin{proof}
    By \cref{lem:H2 of infinite local field},
    $H^2(\G{\Xi/E}, \mu_\ell) \ne 0$ if and only if $\ell^\infty \ndiv [E:F]$, in which case it is isomorphic to $\Z/\ell\Z$.
    The same holds for $H^2(\G{\Xi/E'}, \mu_\ell)$ with the condition $\ell^\infty \ndiv [E':F]$.
    Since $E'/E$ is finite,
    these conditions are equivalent.
\end{proof}

\begin{lem}\label{lem:degree locally implies globally}
    Let $K$ be a number field with a given place $\p$,
    let $\Om/K$ be an algebraic extension and let $n \ge 1$ be an integer.
    If $n \divides [\Om_\fk{P}:K_\p]$ for every $\fk{P} \in \Places_\p(\Om)$
    then $n \divides [\Om:K]$.
    In particular, if $\ell^\infty \divides [\Om_\fk{P}:K_\p]$ 
    for every $\fk{P} \in \Places_\p(\Om)$
    then $\ell^\infty \divides [\Om:K]$.
\end{lem}

\begin{proof}
    Consider the profinite space
    $\Places_\p(\Om)$.
    For every finite subextension $L/K$ of $\Om/K$,
    we have an open subset $U_L \sub \Places_\p(\Om)$ consisting of the places $\fk{P}$
    satisfying $n \divides [L_{\fk{P}|_{L}} : K_\p]$.
    These form an open cover of $\Places_\p(\Om)$, 
    so there is a finite subcover among the $U_L$.
    Since $U_{L_1} \cup U_{L_2} \sub U_{L_1L_2}$,
    there is a single finite subextension $L/K$ of $\Om/K$
    such that $n \divides [L_\fk{P}:K_\p]$
    for every $\fk{P} \in \Places_\p(L)$.
    This implies $[L:K] = \sum_{\fk{P} \divides \p} [L_\fk{P} : K_\p]$
    is also divisible by $n$.
\end{proof}

\begin{lem}\label{lem:indecomposable places are closed}
    Let $E'/E/\Q$ be a tower of algebraic extensions.
    Then the collection of places of $E$ that are indecomposable in $E'$ forms a closed subset of $\Places(E)$.
\end{lem}

\begin{proof}
    Let $\fk{P} \in \Places(E)$ be a place that is decomposable in $E'$.
    Let $\fk{P}_1, \fk{P}_2 \in \Places_{\fk{P}}(E')$ be two distinct places
    and let $K' \sub E'$ be a number field
    such that $\fk{P}_1|_{K'} \ne \fk{P}_2|_{K'}$.
    Let $K = E \cap K'$.
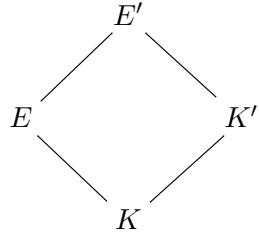
    \[
        \xymatrix{
            & E'\ar@{-}[ld]\ar@{-}[rd] \\
            E\ar@{-}[rd] && K'\ar@{-}[ld] \\
            & K
        }
    \]
    We find that $\fk{P}_1|_{K'}$ and $\fk{P}_2|_{K'}$ 
    are distinct places lying over $\fk{P}|_K$,
    so $\fk{P}|_K$ is decomposable in $K'$.
    Therefore every place of $\Places_{\fk{P}|_K}(E)$
    is decomposable in $E'$.
    The set 
    $\Places_{\fk{P}|_K}(E)$
    is an open neighborhood of $\fk{P}$ in $\Places(E)$, proving that the set of decomposable places is open.
\end{proof}

\begin{prop}\label{prop:infinite indecomposable place}
    Let $\Om/E/\Q$ be a tower of algebraic extensions
    and let $S$ be a finite set of rational places.
    Suppose that for any finite subextension $E'/E$ of $\Om/E$
    there exists $\fk{P} \in S(E)$ which is indecomposable in $E'$.
    Then there exists $\fk{P} \in S(E)$
    which is indecomposable in $\Om$.
\end{prop}

\begin{proof}
    For any finite $E'/E$, the set $C_{E'} \sub S(E)$
    of places that are indecomposable in $E'$
    is closed by \cref{lem:indecomposable places are closed}.
    Considering finite intersections of these sets, we have
    $C_{E'_1} \cap \dots \cap C_{E'_n} \supseteq C_{E'_1 \dots E'_n}$
    and this set is nonempty by assumption.
    The result follows by compactness of $S(E)$.
\end{proof}

We can now prove Neukirch's Theorem for $\ell$-sealed extensions.

\begin{thm}\label{thm:neukirch}
    Let $K$ be a number field with $\Om/K$ an $\ell$-sealed Galois extension.
    Suppose $H \le \G{\Om/K}$ is a subgroup isomorphic to $\G{\Xi/F}$,
    where $F$ is a non-Archimedean local field of characteristic zero,
    and $\Xi/F$ is a Galois extension satisfying $\mu_\ell \sub \Xi$ and $\perfect{\Xi}$.
    Then there exists a unique non-Archimedean place $\fk{P}$ of $\Om$ such that $H \le \G{\Om/K, \fk{P}}$.
\end{thm}

\begin{proof}    
    By \cref{lem:Xi/ell=0 implies ell^infty | [Xi:F]}
    we have $\ell^\infty \divides \card{H}$.
    Thus, uniqueness follows from \cref{prop:local groups ell-disjoint},
    and it suffices to prove the existence of a non-Archimedean $\fk{P}$ with $H \le \G{\Om/K, \fk{P}}$.
    In fact, this cannot hold for any Archimedean place,
    since $\G{\Om/K, \fk{P}}$ has size $1$ or $2$
    for $\fk{P}$ Archimedean, which is not divisible by $\ell^\infty$.
    Thus we ignore the non-Archimedean condition.
    Let $E = \Om^H$. 
    The property $H \le \G{\Om/K, \fk{P}}$
    is equivalent to $\fk{P}|_E$ being indecomposable in $\Om$.
    Therefore, we need to prove there exists a place of $E$
    that is indecomposable in $\Om$.

    Let $S$ be a finite set of rational places
    such that $\Om$ is $\ell$-sealed with respect to it.
    We start by proving the theorem in the case that
    $\mu_\ell \sub K$ and $\mu_\ell \sub F$.
    Since $\G{\Om/E} = H \iso \G{\Xi/F}$, and since $\mu_\ell \iso \Z/\ell\Z$ both as a $\G{\Om/E}$-module and a $\G{\Xi/F}$-module,
    we have
    \[ 
        H^2(\G{\Om/E}, \mu_\ell) \iso H^2(H, \Z/\ell\Z) \iso H^2(\G{\Xi/F}, \mu_\ell) \iso \Z/\ell\Z
    \]
    where the last isomorphism uses \cref{lem:H2 of infinite local field}.

    For every subextension $L/K$ of $E/K$,
    we will denote the group $H^2(\G{\Om/L}, \mu_\ell)$
    by $\cH_L$
    and the group $H^2(\G{\Om^\q/L_\q}, \mu_\ell)$
    by $\cH_{L, \q}$.
    For $L/K$ finite, \cref{thm:brauer sequence} gives rise to a commutative diagram
    \[
        \xymatrix{
            0 \ar[r]& \cH_L \ar[r]\ar[d]& \Op_{\q \in \Places(L)}\cH_{L,\q} \ar[r]\ar[d]& \Z/\ell\Z\ar@{=}[d]\\
            0 \ar[r]& \Br(\Om/L)[\ell]\ar[r]&\Op_{\q \in \Places(L)}\Br(\Om^\q/L_\q)[\ell]\ar[r] &\Z/\ell\Z
        }
    \]
    and taking the direct limit over $K \sub L \sub E$ via restriction maps, we get 
    the commutative diagram with exact rows
    \[
        \xymatrix{
            0 \ar[r] & \cH_E \ar[r]\ar[d] & \dirlim_{K \sub L \sub E}\Op_{\q \in \Places(L)}\cH_{L,\q} \ar[r]\ar[d] & \dirlim_{K \sub L \sub E}\Z/\ell\Z \ar@{=}[d]\\
            0 \ar[r] & \Br(\Om/E)[\ell] \ar[r] & \dirlim_{K \sub L \sub E}\Op_{\q \in \Places(L)}\Br(\Om^\q/L_\q)[\ell] \ar[r] & \dirlim_{K \sub L \sub E}\Z/\ell\Z
        }
    \]
    where the terms on the left are computed by \cref{lem:subgroup cohomology}.
    Recall that $\cH_E \iso \Z/\ell\Z$.
    Further, the limit $\dirlim_{K \sub L \sub E}\Z/\ell\Z$ 
    in the bottom row
    is taken with respect to multiplication by the degree,
    by the classical theory of the Brauer exact sequence.
    It is therefore isomorphic to $\Z/\ell\Z$ if $\ell^\infty \ndiv [E:K]$
    and to $0$ otherwise.
    The same, of course, holds for the limit $\dirlim_{K \sub L \sub E}\Z/\ell\Z$ in the top row.
    It follows that $\dirlim_{K \sub L \sub E}\Op_{\q \in \Places(L)}\cH_{L,\q}$,
    which is a vector space over $\F_\ell$,
    has dimension $2$ if $\ell^\infty \ndiv [E:K]$
    and dimension $1$ otherwise.
    Let $d$ denote its dimension.
    
    For a finite subset $T \sub \Places(E)$,
    there is a homomorphism
    \[
        \pi_T:\dirlim_{K \sub L \sub E}\Op_{\q \in \Places(L)}\cH_{L,\q}
        \to \Op_{\fk{Q} \in T} \cH_{E,\fk{Q}}
    \]
    induced by the restrictions.
    Since the restriction map $T \to \Places(L)$
    is injective for sufficiently large $L$,
    and since $\dirlim_{K \sub L \sub E}\cH_{L,\fk{Q}|_L} \iso \cH_{E, \fk{Q}}$,
    the map $\pi_T$ is surjective.
    In particular, we find that
    $\Op_{\fk{Q} \in T}\cH_{E, \fk{Q}}$
    has dimension at most $d$ for every finite $T \sub \Places(E)$,
    and therefore the same holds for the infinite direct sum
    $\Op_{\fk{Q} \in \Places(E)}\cH_{E, \fk{Q}}$.
    
    It follows in particular that there is a place $\q \in \Places(K) \setminus S(K)$ 
    such that $\cH_{E, \fk{Q}} = 0$ 
    for every $\fk{Q} \in \Places_\q(E)$.
    By \cref{lem:H2 of infinite local field} we see that $\ell^\infty \divides [E_\fk{Q} : K_\q]$
    for every $\fk{Q} \in \Places_\q(E)$,
    and by \cref{lem:degree locally implies globally} this implies $\ell^\infty \divides [E:K]$.
    It follows that $\dirlim_{K \sub L \sub E}\Z/\ell\Z$ vanishes.
    Therefore, $d = 1$. 
    
    Assume, for the sake of contradiction, that $\cH_{E, \fk{P}}$ vanishes for all $\fk{P} \in \Places(E)$.
    We will prove that the direct sum 
    $\dirlim_{K \sub L \sub E}\Op_{\q \in \Places(L)}\cH_{L, \q}$
    also vanishes.
    Let $L/K$ be a finite subextension of $E/K$,
    let $\q \in \Places(L)$ be any place,
    and let $\al \in \cH_{L,\q}$.
    For any finite subextension $L'/L$ of $E/L$,
    consider the subset $U_{L'} \sub \Places_\q(E)$
    consisting of places $\fk{Q}$ such that $\al$ is annihilated
    by the restriction map to $\cH_{L', \fk{Q}|_{L'}}$.
    Those are open subsets, and they form an open cover of $\Places_\q(E)$
    since $\dirlim_{K \sub L' \sub E}\cH_{L',\fk{Q}|_{L'}} \iso \cH_{E,\fk{Q}} = 0$.
    We have $U_{L_1'} \cup U_{L_2'} \sub U_{L_1'L_2'}$,
    and it follows by compactness of $\Places_\q(E)$ 
    that there exists a finite subextension $L'/L$ of $E/L$ 
    such that $\al$ vanishes under the restriction maps to $\Op_{\fk{Q} \in \Places_\q(L')}\cH_{L', \fk{Q}}$.
    This implies that
    $\dirlim_{K \sub L \sub E}\Op_{\q \in \Places(L)}\cH_{L, \q}$
    vanishes, in contradiction to $d = 1$.

    We find that there is a unique $\fk{P}_E \in \Places(E)$ such that 
    $\cH_{E, \fk{P}_E} \ne 0$.
    Note that for any finite subextension $E'/E$ of $\Om/E$,
    the same argument applied to $H' = \G{\Om/E'}$ 
    shows that there is a unique $\fk{P}_{E'} \in \Places(E')$
    such that $\cH_{E', \fk{P}_{E'}} \ne 0$.

    If it were the case that $\fk{P}_E \in \Places(E) \setminus S(E)$,
    then \cref{lem:O/ell=0 outside S} and \cref{cor:H2 unchanged by finite subextension}
    would imply that any place $\tld{\fk{P}_E} \in \Places_{\fk{P}_E}(E')$
    satisfies $\cH_{E', \tld{\fk{P}_E}} \ne 0$.
    By uniqueness of $\fk{P}_{E'}$,
    it follows that $\fk{P}_E$ is indecomposable in $E'$
    for every finite $E \sub E' \sub \Om$,
    and thus in $\Om$.
    This finishes the proof 
    for the case $\mu_\ell \sub K$, $\mu_\ell \sub F$, 
    and $\fk{P}_E \in \Places(E) \setminus S(E)$.

    Now, suppose that $\fk{P}_E \in S(E)$.
    It follows that $\fk{P}_{E'} \in S(E')$ for every finite subextension $E'/E$ of $\Om/E$,
    for otherwise \cref{lem:O/ell=0 outside S} and \cref{cor:H2 unchanged by finite subextension}
    imply that $\cH_{E,\fk{P}_{E'}|_E} \ne 0$,
    contradicting the uniqueness of $\fk{P}_E$.
    
    For $E'/E$ a finite Galois subextension of $\Om/E$,
    and any two places $\fk{P}_1,\fk{P}_2$ of $E'$
    that lie over $\fk{P}_{E'}|_E$,
    we have $\G{\Om^{\fk{P}_1}/E'_{\fk{P}_1}} \iso \G{\Om^{\fk{P}_2}/E'_{\fk{P}_2}}$
    (as isomorphism types of groups) 
    since these are conjugate in $\G{\Om/E}$.
    It follows that $\cH_{E',\fk{P}_1} \iso \cH_{E',\fk{P}_2}$.
    However, $\cH_{E',\fk{P}}$
    is nonzero for $\fk{P} = \fk{P}_{E'}$
    and zero for any other place lying over $\fk{P}_{E'}|_E$.
    It follows that $\fk{P}_{E'}|_E$ is indecomposable in $E'$.

    We have thus shown that for every finite Galois subextension $E'/E$ of $\Om/E$
    there is a place in $S(E)$ that is indecomposable in $E'$.
    This property follows for arbitrary finite subextensions as well,
    by considering their Galois closures.
    It follows from \cref{prop:infinite indecomposable place} that there is a place in $S(E)$ which is indecomposable in $\Om$,
    finishing the proof in the case where
    $\mu_\ell \sub K$ and $\mu_\ell \sub F$.
    
    Now we reduce the general case to the one just proven.
    Let $K' = K(\mu_\ell)$. Then $\G{\Om/K'}$ is an open subgroup of $\G{\Om/K}$.
    Thus, there is a finite subextension $\tld{F}/F$ of $\Xi/F$
    such that $\G{\Xi/\tld{F}} \iso H \cap \G{\Om/K'}$.
    Let $F' = \tld{F}(\mu_\ell)$. 
    Then $\G{\Xi/F'}$ is isomorphic to a subgroup $H' \le \G{\Om/K'}$ which is an open subgroup of $H$.
    By the case just proven, there is a place $\fk{P}$ of $\Om$ such that $H' \le \G{\Om/K', \fk{P}}$.
    
    Let $h \in H$. Both $H'$ and $\cnj{h}{H'}$ are open subgroups of $H$,
    so $H' \cap \cnj{h}{H'}$ is also an open subgroup of $H$.
    In particular it has finite index in $H$,
    and therefore $\ell^\infty \divides \card{H' \cap \cnj{h}{H'}}$.
    Since $H' \sub \G{\Om/K', \fk{P}}$ and $\cnj{h}{H'} \sub \G{\Om/K', h(\fk{P})}$,
    \cref{prop:local groups ell-disjoint}
    implies $h(\fk{P}) = \fk{P}$.
    Since this holds for every $h \in H$, we have $H \le \G{\Om/K, \fk{P}}$ as required.
\end{proof}

\section{Characteristic Detection}\label{sec:char-detect}

\begin{defn}\label{def:abundance}
    We say that an algebraic extension $\Xi/\Q_p$
    is abundant if it contains either $\tm{\Q_p}$ or the $\Z_p^2$-extension of $\Q_p$.
    We say that an algebraic extension $\Om/\Q$ 
        is abundant if, 
        for all rational primes $p$
        outside a set of density zero,
        we have that for all $\fk{P} \in \Places_p(\Om)$,
        the field $\Om_\fk{P}$ is abundant.
\end{defn}

This definition is motivated by the fact that,
for a finite extension $F/\Q_p$
and a Galois extension $\Xi/F$ which is 
abundant,
the group $\G{\Xi/F}$
cannot be embedded into any Galois group of the form 
$\G{\Xi'/F'}$ for $F'$ a local field with a different residue characteristic.
In particular, the group $\G{\Xi/F}$ determines $p$.
This is \cref{prop:local characteristic detection}.

\begin{ex}\label{ex:abundance-examples}
    The fields $\Q(\mu_\infty)$ and $\tm{\Q}$ are abundant
    (and so are algebraic extensions of them).
\end{ex}

\begin{proof}
    The $\Z_p^2$-extension of $\Q_p$ is the compositum of the unramified $\Z_p$-extension
    and the totally wildly ramified $\Z_p$-extension arising from $\mu_{p^\infty}$.
    Both of these extensions are contained in $\Q_p(\mu_\infty)$,
    showing that $\Q(\mu_\infty)$ is abundant.
    To show that $\tm{\Q}$ is abundant, we will show that for every $\fk{P} \in \Places_\fin(\tm\Q)$
    with residue characteristic $p$,
    we have $(\tm{\Q})_\fk{P} \iso \tm{\Q_p}$.
    This follows from the next lemma.
    \end{proof}

\begin{lem}
    Let $K$ be a number field, let $\p \in \Places(K)$,
    and let $F/K_\p$ be a finite tamely ramified extension of degree $d$.
    Then there is a finite tamely ramified extension $L/K$
    of degree $d$,
    in which $\p$ is indecomposable,
    such that there is an isomorphism $F \iso L^\p$
    of $K_\p$-algebras.
    Here $L^\p$ denotes the completion of $L$ at the unique place lying over $\p$.
\end{lem}

\begin{proof}
    If $\p$ is a complex Archimedean place,
    we can take $L = K$.
    If $\p$ is a real Archimedean place,
    we can take $L = K(\mu_3)$.
    Thus we may assume that $\p$ is a non-Archimedean place.
    By the Primitive Element Theorem, there is $\al \in F$ such that $F = K_\p(\al)$.
    Let $f_\p \in K_\p[X]$ be the minimal polynomial of $\al$,
    so that $\deg(f_\p) = d$ and $F \iso K_\p[X]/(f_\p)$. 
    Let $S$ be the set of rational primes less than or equal to $d$.
    For every prime $\q \in S(K) \setminus \{\p\}$,
    let $f_\q \in K_\q[X]$ be a monic irreducible polynomial of degree $d$
    such that $K_\q[X]/(f_\q)$
    is isomorphic to the unramified extension of $K_\q$ of degree $d$.
    By the Chinese Remainder Theorem,
    there is a monic polynomial $f \in K[X]$ of degree $d$
    that is arbitrarily close to $f_\q$ in $K_\q[X]$
    for every $\q \in S(K)$.
    By Krasner's lemma,
    we may assume that $f$ is irreducible in $K_\q[X]$
    and that $K_\q[X]/(f) \iso K_\q[X]/(f_\q)$
    for every $\q \in S(K)$.
    Let $L = K[X]/(f)$.
    Then $L/K$ is a degree $d$ extension
    in which $\p$ is indecomposable
    with $L^\p \iso F$,
    and moreover, $L/K$
    is unramified (and in particular tamely ramified)
    at every prime of $S(K) \setminus \{\p\}$.
    Consider $\tld{L}$, the Galois closure of $L/K$.
    Since the compositum of two tamely ramified extensions is tamely ramified,
    it follows that $\tld{L}/K$ is tamely ramified at $S(K)$.
    However, $\tld{L}/K$ is a Galois extension of degree dividing $d!$,
    so it is tamely ramified at all prime numbers not dividing $d!$,
    which are exactly those outside $S(K)$.
    This shows that $\tld{L}/K$ is tamely ramified, and thus $L/K$ is tamely ramified as well.
    This completes the proof.
\end{proof}

\begin{rmk}
    The fields $\Om$ in \cref{ex:abundance-examples} satisfy the property that $\Om_\fk{P}$
    is abundant for every $\fk{P} \in \Places_\fin(\Om)$,
    even though the definition only requires this for $\fk{P}$
    lying over a set of rational primes with density one.
\end{rmk}

The next lemma collects standard facts about
$p$-Sylow subgroups of profinite groups.
\begin{lem}\label{lem:pro-p lift}
    Let $\pi : G' \to G$ be a surjection of profinite groups.
    Let $p$ be a prime number.
    \begin{enumerate}
        \item If $P \le G$ is a pro-$p$ subgroup,
        then there exists a pro-$p$ subgroup $P' \le G'$ such that $\pi(P') = P$.
        \item If $P' \le G'$ is a $p$-Sylow subgroup
        then $\pi(P') \le G$ is a $p$-Sylow subgroup.
        \item If $p \ndiv \card{\ker(\pi)}$ then the $p$-Sylow subgroups of $G'$ 
        are isomorphic to those of $G$.
    \end{enumerate}
\end{lem}

\begin{proof}
    By \cite[2.3.6]{profinite groups ribes zalesskii},
    any profinite group has a $p$-Sylow subgroup,
    and the $p$-Sylow subgroups are all conjugate,
    so the isomorphism type of a $p$-Sylow subgroup is well-defined.

    For the first part, let $P'$ be a $p$-Sylow subgroup of $\pi\inv(P)$.
    Since $[P : \pi(P')] \divides [\pi\inv(P) : P']$ and $p \ndiv [\pi\inv(P) : P']$,
    we find that $p \ndiv [P : \pi(P')]$, so $\pi(P') = P$.
    For the second part, similarly, we have $[G : \pi(P')] \divides [G' : P']$ and $p \ndiv [G' : P']$,
    so that $p \ndiv [G : \pi(P')]$ and $\pi(P') \le G$ is a $p$-Sylow subgroup.
    For the third part, let $P' \le G'$ be a $p$-Sylow subgroup.
    Then $\pi|_{P'}:P' \to \pi(P')$ is a surjection from a $p$-Sylow subgroup of $G'$
    to a $p$-Sylow subgroup of $G$. However, $\ker(\pi|_{P'}) = \ker(\pi) \cap P'$
    is trivial by $p \ndiv \card{\ker(\pi)}$, so $\pi|_{P'}$ is actually an isomorphism.
\end{proof}

For a finite extension $F/\Q_p$
with a Galois extension $\Xi/F$,
we define $I_{\Xi/F} = \G{\Xi/\Xi \cap \un F}$ (the inertia subgroup)
and $W_{\Xi/F} = \G{\Xi/\Xi \cap \tm F}$ (the wild ramification subgroup).

Suppose $r$ is a prime number, $n$ is a positive integer
not divisible by $r$,
and $Z$ is a procyclic group
with a distinguished topological generator $1 \in Z$
(for instance $\hat\Z$ or $\Z_p$),
such that the order of $n$ in $\Z_r^\x \iso \Aut(\Z_r)$ divides $\card{Z}$. 
Then we write $\Z_r \rtimes_n Z$ for the semidirect product
where $1 \in Z$ acts on $\Z_r$ by multiplication with $n$.

\begin{lem}\label{lem:Zr semidirect structure}
    Let $F$ be a finite extension of $\Q_p$ with residue field of size $q$,
    let $\Xi/F$ be a Galois extension and let $r \ne p$ be a prime number. Then:
    \begin{enumerate}
        \item 
        The $r$-Sylow subgroups of $I_{\ov F/F}$ are isomorphic to $\Z_r$
        and the $r$-Sylow subgroups of $\G{\ov F/F}$
        are isomorphic to $\Z_r \rtimes_{q^{r-1}} \Z_r$.
        \item If $\Xi$ contains $\tm{\Q_p}$
        then there is an injection $\Z_r \rtimes_q \hat\Z \to \G{\Xi/F}$.
        \item If $n > 1$ is an integer coprime to $r$ such that there is an injection $\Z_r \rtimes_n \hat\Z \to \G{\Xi/F}$,
        then there is $u \in \hat\Z$
        such that $n = q^u$ in $\Z_r$.
    \end{enumerate}
\end{lem}

\begin{proof}
    For the first part,
    note that the wild ramification subgroup $W_{\ov F/F}$,
    contained in $I_{\ov F/F}$ and normal in $\G{\ov F/F}$,
    is a pro-$p$ group.
    Thus, from \cref{lem:pro-p lift} it follows that
    the $r$-Sylow subgroups of $I_{\ov F/F}$ and $\G{\ov F/F}$ 
    are isomorphic to those of $I_{\tm F/F}$ and $\G{\tm F/F}$ respectively.

    Since $I_{\tm F/F} \iso \prod_{r' \ne p} \Z_{r'}$,
    the $r$-Sylow subgroups of $I_{\tm F/F}$ are isomorphic to $\Z_r$.
    Further, the subgroup $N \le \G{\tm F / \un F}$ corresponding to $\prod_{r' \ne p,r}\Z_{r'}$
    is normal in $\G{\tm F/F}$
    and satisfies $r \ndiv \card{N}$. 
    By \cref{lem:pro-p lift}, the $r$-Sylow subgroups of $\G{\tm F/F}$
    are isomorphic to those of $\G{\tm F/F}/N \iso \Z_r\rtimes_q\hat\Z$.
    Let $\tau,\s \in \Z_r\rtimes_q\hat\Z$
    be topological generators of the $\Z_r$ part and the $\hat\Z$ part respectively,
    with $\cnj{\s}{\tau} = \tau^q$.
    Let $u \in \hat\Z$ be the profinite integer which is equivalent to $r - 1$ in $\Z_r$
    and to $0$ in $\Z_{r'}$ for every prime $r' \ne r$.
    Then $\s^u$ generates the $r$-Sylow subgroup of $\hat\Z$,
    and since $u \equiv r - 1 \pmod{(r-1)r^\infty}$
    (which is the order of $\Z_r^\x$)
    we have $\cnj{\s^u}{\tau} = \cnj{\s^{r - 1}}{\tau} = \tau^{q^{r-1}}$.
    It follows that $\ang{\tau, \s^u} \iso \Z_r\rtimes_{q^{r - 1}}\Z_r$.
    Since this is a normal subgroup of $\Z_r\rtimes_q\hat\Z$
    and the quotient by it is coprime to $r$,
    it is the unique $r$-Sylow subgroup.
    This proves the first part.
    
    For the second part,
    consider the commutative diagram
    \[
        \xymatrix{
            0 \ar[r] & I_{\Xi/F} \ar[r]\ar[d]^{\phi_I} & \G{\Xi/F} \ar[r]\ar[d]^{\phi_G} & \hat\Z \ar[r]\ar@{_(->}[d]^{\cdot  f} & 0\\
            0 \ar[r] & I_{\tm{\Q_p}/\Q_p} \ar[r] & \G{\tm{\Q_p}/\Q_p} \ar[r] & \hat\Z \ar[r] & 0
        }
    \]
    where $f = f_{F/\Q_p} = \log_pq$,
    and $\phi_I,\phi_G$ have open images
    (of indices $e_{F \cap \tm{\Q_p}/\Q_p}$ and $[F \cap \tm{\Q_p} : \Q_p]$ respectively).

    Let $R \le I_{\Xi/F}$ be an $r$-Sylow subgroup.
    Consider the surjection $I_{\ov F/F} \to I_{\Xi/F}$.
    The $r$-Sylow subgroup of $I_{\ov F/F}$
    is isomorphic to $\Z_r$,
    so by \cref{lem:pro-p lift}, $R$ is isomorphic to a quotient of $\Z_r$.
    However, the $r$-Sylow subgroup of $I_{\tm{\Q_p}/\Q_p}$ is also isomorphic to $\Z_r$,
    so $\phi_I(R)$ is isomorphic to an open subgroup of $\Z_r$.
    This shows that $R \iso \Z_r$.

    Let $\s \in \G{\Xi/F}$ be a preimage of the Frobenius element $1 \in \hat\Z$
    and let $\tau$ be a generator of $R$.
    Note that $\phi_I(\cnj{\s}{\tau}) = \phi_I(\tau)^q$.
    Since $R, \cnj\s{R}$ are two $r$-Sylow subgroups of $I_{\Xi/F}$,
    there is $x \in I_{\Xi/F}$
    such that $\cnj{x\s}{R} = R$.
    Denote $\s' = x\s$.
    Then we have $\cnj{\s'}{\tau} \in R$ and also $\phi_I(\cnj{\s'}{\tau}) = \phi_I(\tau)^q$.
    Since $R \iso \Z_r$
    we find that $\phi_I$ is injective on $R$,
    implying that $\cnj{\s'}{\tau} = \tau^q$.
    Therefore, $\ang{\tau, \s'} \iso \Z_r \rtimes_q \hat\Z$,
    proving the second part.

    For the third part,
    suppose that $\Z_r \rtimes_n \hat\Z \iso \ang{a, b \mid a^{r^\infty} = e, \cnj{b}{a} = a^n}$
    injects into $\G{\Xi/F}$, and denote the images of $a$, $b$
    in $\G{\Xi/F}$ by $\tau$, $\s$ respectively.

    Consider the short exact sequence $0 \to I_{\Xi/F} \to \G{\Xi/F} \arw j \G{\Xi/F}/I_{\Xi/F} \to 0$.
    Since the codomain of $j$ is abelian,
    we have $j(\tau) = j(\cnj{\s}{\tau}) = j(\tau)^n$,
    implying that $j(\tau^{n - 1}) = e$.
    Note that there is another injection $\Z_r \rtimes_n \hat\Z \to \G{\Xi/F}$,
    sending $a \mapsto \tau^{n - 1}$ and $b \mapsto \s$.
    Thus, we may assume without loss of generality that $\tau \in I_{\Xi/F}$.
    
    Let $\s_0 \in \G{\Xi/F}$ denote a preimage of the Frobenius element of $\G{\Xi/F}/I_{\Xi/F}$.
    Then there is $u \in \hat\Z$ such that $\s \cdot \s_0^{-u} \in I_{\Xi/F}$.
    Consider the quotient map $\pi:\G{\Xi/F} \to \G{\Xi/F}/W_{\Xi/F}$.
    Note that $\pi(\cnj{\s_0}{x}) = \pi(x)^q$ for every $x \in I_{\Xi/F}$,
    and therefore $\pi(\tau)^n = \pi(\cnj\s{\tau}) = \pi(\cnj{(\s \s_0^{-u}) \s_0^u}{\tau}) =  \pi(\tau)^{q^u}$.
    Since the kernel of $\pi$ is a pro-$p$ group,
    its restriction to $\ang{\tau}$ is injective,
    meaning that $\tau^n = \tau^{q^u}$,
    and thus $n = q^u$ in $\Z_r$.
\end{proof}

\begin{prop}\label{prop:local characteristic detection}
    Let $F,F'$ be non-Archimedean local fields of characteristic zero,
    and let $\Xi/F, \Xi'/F'$ be Galois extensions, with $\Xi$ abundant.
    Suppose there is a profinite group embedding $\al:\G{\Xi/F} \to \G{\Xi'/F'}$.
    Then $F$ and $F'$ have the same residue characteristic.
\end{prop}

\begin{proof}
    Denote the residue characteristics of $F, F'$ by $p, p'$,
    and denote their residue field sizes by $q, q'$.
    Assume, for the sake of contradiction, that $p \ne p'$.
    
    First, suppose that $\Xi$ contains the $\Z_p^2$-extension $E/\Q_p$.
    Consider the natural continuous homomorphism $\phi:\G{EF/F} \to \G{E/\Q_p} \iso \Z_p^2$.
    Since $[F : \Q_p] < \infty$, it follows that $\phi$ has open image.
    As open subgroups of $\Z_p^2$ are isomorphic to $\Z_p^2$,
    we find that $\G{EF/F}$, and therefore $\G{\Om/F}$, has a quotient isomorphic to $\Z_p^2$.
    By \cref{lem:pro-p lift} it follows that there is a pro-$p$ subgroup $P_0 \le \G{\Xi/F}$
    with a quotient isomorphic to $\Z_p^2$.
    Applying \cref{lem:pro-p lift} again to $\al(P_0) \le \G{\Xi'/F'}$
    and to the surjection $\G{\ov{F'}/F'}\to\G{\Xi'/F'}$,
    we find that there is a pro-$p$ subgroup $P \le \G{\ov{F'}/F'}$
    with a quotient isomorphic to $\Z_p^2$.
    Let $\tld{P} \le \G{\ov{F'}/F'}$ be a $p$-Sylow subgroup containing $P$.
    Since $p \ne p'$,
    we have $\tld{P} \iso\Z_p \rtimes_{(q')^{p-1}} \Z_p$
    by the first part of \cref{lem:Zr semidirect structure}.
    
    Let $\tau, \s \in \tld{P}$ denote the generators corresponding to the semidirect product structure,
    so that 
    $\ang{\tau} \iso \ang{\s} \iso \Z_p$, $\cnj\s{\tau} = \tau^{(q')^{p-1}}$, $\ang{\tau, \s} = \tld{P}$ and $\ang{\tau} \cap \ang{\s} = \{e\}$.
    Consider the quotient map $\pi :\tld{P} \to \tld{P}/\ang{\tau}$.
    Since $P$ has a quotient isomorphic to $\Z_p^2$,
    it cannot be contained in the procyclic group $\ang{\tau}$,
    so $\pi(P) \ne \{e\}$.
    As $\tld{P}/\ang{\tau} = \ang{\pi(\s)} \iso \Z_p$,
    there is an integer $k \ge 0$ such that $\pi(P) = \ang{\pi(\s^{p^k})}$.
    Let $\Sigma \in P$ be a preimage of $\pi(\s^{p^k})$
    and let $T \in P$ be a generator of the procyclic group $P \cap \ang{\tau}$.
    Then we have $P = \ang{T, \Sigma}$ and $\cnj{\Sigma}{T} = T^{(q')^{(p - 1)p^k}}$.
    It follows that $P^\ab$
    is a quotient of $\Z_p/((q')^{(p - 1)p^k} - 1) \x \Z_p$.
    This group does not have a quotient isomorphic to $\Z_p^2$,
    which is a contradiction. This shows $p = p'$ in the case where $\Xi$ contains the $\Z_p^2$-extension of $\Q_p$.

    Now assume $\Xi$ contains $\tm{\Q_p}$.
    Let $\ell$
    be a prime number not dividing
    $p \cdot \log_pq$.
    Consider the number fields
    $K = \Q(\mu_\ell, (q')^\frac1\ell)$ and $L = \Q(\mu_\ell, (q')^\frac1\ell, q^\frac1\ell)$.
    Note that $L$ is ramified at $p$ while $K$ is not, and therefore $K \subsetneq L$.
    Since $K$ and $L$ are Galois over $\Q$, Chebotarev's Density Theorem implies there are infinitely many rational primes
    that split completely in $K$ but not in $L$.
    Take such a prime $r \ne p, p'$.
    The splitness condition is equivalent to $r \equiv 1 \pmod{\ell}$,
    with $q'$ being an $\ell$th power modulo $r$,
    and $q$ not being an $\ell$th power modulo $r$.
    By the second part of \cref{lem:Zr semidirect structure}
    we find that $\Z_r \rtimes_q \hat\Z$ injects into $\G{\Xi/F}$,
    so it also injects into $\G{\Xi'/F'}$,
    and by the third part of \cref{lem:Zr semidirect structure}
    we find that $q$ lies in the procyclic subgroup of $\Z_r^\x$ generated by $q'$.
    However, this is impossible since $q'$ is an $\ell$th power modulo $r$
    and $q$ is not. The contradiction shows $p = p'$
    in the case where $\Xi$ contains $\tm{\Q_p}$,
    completing the proof.
\end{proof}

As a corollary we obtain a strengthened version of \cref{thm:neukirch}
for extensions which are $\ell$-sealed and abundant.

\begin{cor}\label{cor:char detect neukirch}
    With the notation of \cref{thm:neukirch},
    assume that $\Xi$ is abundant.
    Then $F$ and $\fk{P}$ have the same residue characteristic.
\end{cor}

\begin{proof}
    Since $\G{\Xi/F} \iso H \le \G{\Om/K, \fk{P}} \iso \G{\Om_\fk{P}/K_{\fk{P}|_K}}$,
    this follows from \cref{prop:local characteristic detection}.
\end{proof}

\section{Prime Bijection}\label{sec:prime-bijection}

\begin{defn}
    Suppose that $\Om_1/K_1$ and $\Om_2/K_2$ are Galois extensions
    and that $\al:\G{\Om_1/K_1} \to \G{\Om_2/K_2}$ is an isomorphism of profinite groups.
    Then Galois correspondence implies that there is an inclusion-preserving bijection between the
    lattice of subextensions of $\Om_1/K_1$ and the lattice of subextensions of $\Om_2/K_2$
    which we also denote by $\al$. 
\end{defn}

Note that $\al(L \cap L') = \al(L) \cap \al(L')$, $\al(L \cdot L') = \al(L) \cdot \al(L')$, $\al(\G{\Om_1/L}) = \G{\Om_2/\al(L)}$, $\al(g(L)) = \al(g)(\al(L))$,
and $[\al(L') : \al(L)] = [L' : L]$.

\begin{prop}\label{prop:prime bijection}
    Let $K_1, K_2$ be number fields,
    and let $\Om_i/K_i$ be Galois extensions
    that are $\ell$-sealed with respect to $S$
    and abundant,
    for some finite set $S$ of rational primes.
    Suppose that $\al:\G{\Om_1/K_1} \to \G{\Om_2/K_2}$
    is an isomorphism of profinite groups.
    Then there is a set $R$ of rational primes with full density, disjoint from $S$,
    and a family of bijections $R(L) \to R(\al(L))$
    over the subextensions $\Om_1/L/K_1$,
    which we denote by $\al$, satisfying the following properties:

    \begin{enumerate}
        \item If $L/K_1$ is Galois,
        then for $\p \in R(L)$
        we have $\al(\G{L/K_1, \p}) = \G{\al(L)/K_2, \al(\p)}$.
        \item The bijections preserve residue characteristics.
        \item For $g \in \G{\Om_1/K_1}$
        and $\p \in R(L)$,
        we have $\al(g\p) = \al(g)\al(\p)$ (as primes of $\al(g(L)) = \al(g)(\al(L))$).
        \item For subextensions $\Om_1/L'/L/K_1$ and $\p \in R(L')$
        we have $\al(\p|_L) = \al(\p)|_{\al(L)}$
        and $[\al(L')_{\al(\p)} : \al(L)_{\al(\p|_L)}] = [L'_\p : L_{\p|_L}]$.
    \end{enumerate}
\end{prop}

\begin{proof}
    Let $R$ be the full density set of rational primes $p$
    outside $S$ such that for every $\fk P_1 \in \Places_p(\Om_1)$
    and $\fk P_2 \in \Places_p(\Om_2)$, the fields $(\Om_1)_{\fk P_1}$
    and $(\Om_2)_{\fk P_2}$ are abundant. 
    We start by constructing a bijection for $L = \Om_1$
    satisfying the first three conditions.
    Let $\fk{P} \in R(\Om_1)$.
    We have $\mu_\ell \sub (\Om_1)_\fk{P}$
    and $\perfect{(\Om_1)_{\fk{P}}}$ by \cref{lem:O/ell=0 outside S}.
    Thus, \cref{thm:neukirch}
    applied to the 
    closed subgroup $\al(\G{\Om_1/K_1, \fk{P}}) \le \G{\Om_2/K_2}$
    implies that there is a unique non-Archimedean place $\fk{Q}$ of $\Om_2$ such that
    $\al(\G{\Om_1/K_1, \fk{P}}) \sub \G{\Om_2/K_2, \fk{Q}}$.
    Then, \cref{cor:char detect neukirch} implies $\fk{Q}$ and $\fk{P}$ have the same residue characteristic,
    which in particular means that $\fk{Q} \in R(\Om_2)$.
    This defines a map $\al:R(\Om_1) \to R(\Om_2)$
    which preserves residue characteristics
    and has the property $\al(\G{\Om_1/K_1, \fk{P}}) \sub \G{\Om_2/K_2, \al(\fk{P})}$.

    Similarly we can define a map $\al\inv:R(\Om_2) \to R(\Om_1)$
    satisfying $\al\inv(\G{\Om_2/K_2, \fk{Q}}) \sub \G{\Om_1/K_1, \al\inv(\fk{Q})}$
    for every $\fk{Q} \in R(\Om_2)$.
    Composing them gives $\G{\Om_1/K_1, \fk{P}} \sub \G{\Om_1/K_1, \al\inv(\al(\fk{P}))}$.
    By the uniqueness in \cref{thm:neukirch}, this implies $\al\inv(\al(\fk{P})) = \fk{P}$.
    Similarly we have $\al(\al\inv(\fk{Q})) = \fk{Q}$ for every $\fk{Q} \in R(\Om_2)$,
    and thus $\al$ and $\al\inv$ are inverse bijections.
    We also get $\al(\G{\Om_1/K_1, \fk{P}}) = \G{\Om_2/K_2, \al(\fk{P})}$.
    It follows that $\al : R(\Om_1) \to R(\Om_2)$
    satisfies the first and second properties.
    For the third property, note that
    \[
        \al(\G{\Om_1/K_1, g\fk{P}}) = \al\pr{\cnj{g}{\G{\Om_1/K_1, \fk{P}}}} = \cnj{\al(g)}{\al(\G{\Om_1/K_1, \fk{P}})} = \cnj{\al(g)}{\G{\Om_2/K_2, \al(\fk{P})}} = \G{\Om_2/K_2, \al(g)(\al(\fk{P}))}
    \]
    and use the uniqueness in \cref{thm:neukirch}.

    Now, let $\Om_1/L/K_1$ be a subextension and let $\p \in R(L)$.
    Let $\fk{P}_1, \fk{P}_2 \in \Places_\p(\Om_1)$.
    Then there is some $g \in \G{\Om_1/L}$
    satisfying $g\fk{P}_1 = \fk{P}_2$,
    so $\al(\fk{P}_2) = \al(g)\al(\fk{P}_1)$. Since $\al(g)|_{\al(L)} = \id_{\al(L)}$,
    this implies $\al(\fk{P}_2)|_{\al(L)} = \al(\fk{P}_1)|_{\al(L)}$.
    Thus, we get a well-defined map
    $\al : R(L) \to R(\al(L))$
    which takes $\p$ to $\al(\fk{P})|_{\al(L)}$ for any $\fk{P} \in \Places_{\p}(\Om_1)$.
    For general $L$,
    the first three properties
    follow by restricting from the $L = \Om_1$ case.
    For the fourth property, it is clear from the definition that $\al(\p|_L) = \al(\p)|_{\al(L)}$.
    Let $\fk{P} \in \Places_\p(\Om_1)$.
    Then
    \[
        [\al(L')_{\al(\p)} : \al(L)_{\al(\p|_L)}]
        = [\G{\Om_2/\al(L), \al(\fk{P})} : \G{\Om_2/\al(L'), \al(\fk{P})}]
        = [\G{\Om_1/L, \fk{P}} : \G{\Om_1/L', \fk{P}}]
        = [L'_\p : L_{\p|_L}].
    \]
\end{proof}

\begin{cor}\label{cor:arith equiv}
    Let $K_1, K_2$ be number fields
    and let $\Om_i/K_i$ be $\ell$-sealed and abundant Galois extensions. 
    Suppose $\al:\G{\Om_1/K_1} \to \G{\Om_2/K_2}$
    is an isomorphism of profinite groups.
    Then $K_1$ and $K_2$ are arithmetically equivalent.
\end{cor}

\begin{proof}
    By \cref{prop:prime bijection},
    there is a set $R$ of rational primes with full density
    and a bijection $\al : R(K_1) \to R(K_2)$
    that preserves residue characteristics.
    In particular, $K_1$ and $K_2$ have the same number of primes lying over $p$
    for every $p \in R$.
    By the main theorem of \cite{arithmetic equivalence by number of primes},
    this implies $K_1$ and $K_2$ are arithmetically equivalent.
\end{proof}

\section{Weak Neukirch-Uchida}\label{sec:uchida-isomorphism}

Let $G$ be a finite group.
We consider $\F_\ell[G]^n$ as a left $G$-module,
and we denote its elements by $\sum_{i=1}^n \sum_{g \in G} \lambda_{ig}ge_i$ for $\lambda_{ig} \in \F_\ell$.
For a $G$-module $M$ and a subgroup $H \le G$
we define $I_HM$ as the $H$-submodule of $M$
generated by the elements $hm-m$ for all $h \in H, m \in M$.

When $\tld{L}/L/\Q$ is a tower of number fields
with both $\tld{L}/\Q$, $L/\Q$ Galois and $\tld{L}/L$ abelian,
the abelian group $\G{\tld{L}/L}$ is naturally a $\G{L/\Q}$-module.

\begin{lem} \label{lem:embedding problem}
    Let $\ell$ be a prime number and $S$ a finite set of rational places.
    Suppose $L/\Q$ is a finite Galois extension containing $\mu_\ell$
    and let $G = \G{L/\Q}$. 
    Then for every $n \in \N$
    there exists an abelian extension $\tld{L}/L$, 
    in which every place of $S(L)$
    splits completely,
    such that $\tld{L}/\Q$ is Galois 
    and $\G{\tld{L}/L} \iso \F_\ell[G]^n$ as a $G$-module. 
    Furthermore, $\tld{L}$ can be chosen such that
    for every subfield $K \sub L$ containing $\mu_\ell$
    there exists an abelian subextension $\tld{K}/K$ of $\tld{L}/K$,
    of exponent $\ell$,
    in which every place of $S(K)$ splits completely,
    and which satisfies
    $\G{\tld{L}/L\tld{K}} = I_H\G{\tld{L}/L}$
    for $H = \G{L/K}$.
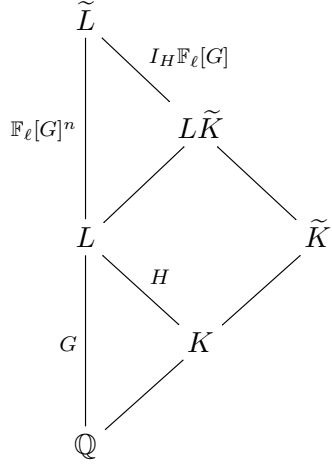
    \[
        \xymatrix{
            \tld{L}\ar@{-}[dr]^{I_H\F_\ell[G]}\ar@{-}[dd]_{\F_\ell[G]^n}\\
            & L\tld{K}\ar@{-}[dl]\ar@{-}[dr] \\
            L \ar@{-}[dr]^H\ar@{-}[dd]_G && \tld{K}\ar@{-}[dl] \\
            & K \ar@{-}[dl] \\
            \Q
        }
    \]
\end{lem}

\begin{proof}
    By Chebotarev's Density Theorem,
    there are infinitely many prime numbers that split completely in $L$.
    Let $T$ be a set consisting of $n$ such prime numbers
    that do not belong to $S$.
    Then $T(L)$ is a free $G$-set of cardinality $n \cdot \card{G}$.

    Let $\fk{P}_1, \dots, \fk{P}_n$
    be representatives for the $G$-orbits in $T(L)$.
    For each $1 \le k \le n$,
    \cref{prop:K^x/K^ell local surjection}
    for the set $S \cup T$
    implies there exists $\al_{\fk{P}_k} \in L^\x$
    whose images in $L_\fk{P}, \fk{P} \in S(L) \cup T(L) \setminus \{\fk{P}_k\}$
    are $\ell$th powers, while its image in $L_{\fk{P}_k}$ is not.
    Extend the function $\fk{P}_k \mapsto \al_{\fk{P}_k}$
    to a $G$-set morphism $\fk{P} \mapsto \al_\fk{P}$
    from $T(L)$ to $L^\x$
    and denote $L(\al_\fk{P}^\frac1\ell)$ by $L_{(\fk{P})}$.
    Then every place in $S(L)\cup T(L) \setminus \{\fk{P}\}$ splits completely in $L_{(\fk{P})}$
    while $\fk{P}$ is indecomposable.
    Also, note that $G$ acts on the set of abelian extensions of $L$,
    and under this action $g(L_{(\fk{P})}) = L_{(g\fk{P})}$. 
    This action also defines isomorphisms $\G{L_{(\fk{P})}/L} \arw{\phi_g} \G{L_{(g\fk{P})}/L}$
    satisfying $\phi_e=\id, \phi_{gg'} = \phi_g \circ \phi_{g'}$ for all $g, g' \in G$.

    Consider the compositum $\tld{L} = \prod_{\fk{P} \in T(L)}L_{(\fk{P})}$, which is Galois over $\Q$.
    Since $L_{(\fk{P})}$ is the only extension among the $L_{(\fk{P}')}$ where $\fk{P}$ is indecomposable,
    it follows that $L_{(\fk{P})} \nsubseteq \prod_{\fk{P}' \ne \fk{P}} L_{(\fk{P}')}$,
    so $[\tld{L}:L] = \ell^{\card{T(L)}}$.
    Therefore, the injection $\G{\tld{L}/L} \to \prod_{\fk{P} \in T(L)} \G{L_{(\fk{P})}/L}$ is an isomorphism.
    However, this is a $G$-equivariant map,
    where the action of $G$ on the right-hand side is given by
    $g \cdot (\lambda_\fk{P})_{\fk{P} \in T(L)} = (\phi_g(\lambda_{g\inv\fk{P}}))_{\fk{P} \in T(L)}$.
    This action makes the right hand side isomorphic to $\F_\ell[G]^n$,
    proving the first part of the lemma.
    
    Let $K \sub L$ be a subfield containing $\mu_\ell$
    and let $H = \G{L/K}$.
    For $\p \in T(K)$,
    define $\be_{\p} = \Nm_K^L\al_\fk{P}$ for any $\fk{P} \in \Places_\p(L)$,
    and define $K_{(\p)} = K(\be_{\p}^\frac1\ell)$.
    
    For every $\fk{P} \in T(L), \fk{Q} \in S(L) \cup T(L)$ denote by $\al_{\fk{P};\fk{Q}}$
    the image of $\al_\fk{P}$ in the completion $L_{\fk{Q}}$
    and for every $\p \in T(K), \q \in S(K) \cup T(K)$
    denote by $\be_{\p; \q}$ the image of $\be_\p$ in the completion $K_\q$.
    Then $\be_{\p;\q} = \prod_{\fk{Q} \in \Places_\q(L)}\Nm_{K_\q}^{L_{\fk{Q}}}\al_{\fk{P};\fk{Q}}$
    for any $\fk{P} \in \Places_\p(L)$.
    Since $\Nm_{K_\q}^{L_{\fk{Q}}} : L_{\fk{Q}} \to K_\q$ is an isomorphism,
    and since $\al_{\fk{P};\fk{Q}}$
    is an $\ell$th power if and only if $\fk{P} \ne \fk{Q}$,
    we find that $\be_{\p;\q}$ is an $\ell$th power 
    if and only if $\p \ne \q$.
    It follows that all places of $S(K) \cup T(K) \setminus \{\p\}$
    split completely in $K_{(\p)}/K$,
    while $\p$ is indecomposable.
    Since all of them split completely in $L$,
    it follows that
    all the places in $S(L)\cup T(L) \setminus \Places_\p(L)$
    split completely in $K_{(\p)}L$
    while the places in $\Places_\p(L)$
    are indecomposable in $K_{(\p)}L$. 
    Consider the compositum $\tld{K} = \prod_{\p \in T(K)}K_{(\p)}$.
    This is an abelian extension of $K$ of exponent $\ell$. 
    From the above, we find that $[\tld{K}:K] = [\tld{K}L:L] = \ell^{\card{T(K)}}$,
    and in particular $\tld{K} \cap L = K$.

    Consider the $H$-submodule $I_H\G{\tld{L}/L} \le \G{\tld{L}/L}$.
    It is equal to the commutator group $[\G{\tld{L}/K}, \G{\tld{L}/L}]$.
    The elements of this group act trivially on $L$,
    and they also act trivially on $\tld{K}$ since $\tld{K}/K$ is abelian.
    Thus we have $I_H\G{\tld{L}/L} \sub \G{\tld{L}/L\tld{K}}$.
    We also have $[\tld{L}:L\tld{K}] = \frac{[\tld{L}:L]}{[L\tld{K}:L]} = \frac{[\tld{L}:L]}{[\tld{K}:K]} = \card{I_H\G{\tld{L}/L}}$,
    showing that $I_H\G{\tld{L}/L} = \G{\tld{L}/L\tld{K}}$.
\end{proof}

\begin{thm} \label{thm:uchida isomorphism}
    Let $K_1, K_2$ be number fields
    containing $\mu_\ell$,
    let $\Om_i/K_i$ be $\ell$-sealed and abundant Galois extensions,
    and suppose there exists an isomorphism $\al:\G{\Om_1/K_1} \to \G{\Om_2/K_2}$.
    Then $K_1$ and $K_2$ are isomorphic.
\end{thm}

\begin{proof}
    Let $L$ be the Galois closure of $K_1K_2$ over $\Q$.
    Denote $G = \G{L/\Q}$ and $H_i = \G{L/K_i} \le G$.
    Let $n = \card{H_1}$.
    Applying \cref{lem:embedding problem}
    with $S$ such that both $\Om_i$ are $\ell$-sealed with respect to it,
    we obtain an extension $\tld{L}/L$
    with Galois group isomorphic to $\F_\ell[G]^n$ as a $G$-module,
    in which every place of $S(L)$ splits completely. 
    We also obtain abelian subextensions $\tld{K}_i/K_i$ of $\tld{L}/K_i$,
    of exponent $\ell$,
    in which every place of $S(K_i)$ splits completely,
    and which satisfy $\G{\tld{L}/L\tld{K}_i} = I_{H_i}\F_\ell[G]^n$.
    It follows that $\tld{K}_i \sub \Om_i$. 

    \[
        \xymatrix@C+1.3pc@R+1.3pc{
            \Om_1\ar@{-}[dd]&&\tld{L}\ar@{-}[dd]^{\F_\ell[G]^n} \ar@{-}[ld]_{I_{H_1}\F_\ell[G]^n}\ar@{-}[rd]^{I_{H_2}\F_\ell[G]^n} && \Om_2\ar@{-}[dd]\\
            &L\tld{K}_1\ar@{-}[rd]\ar@{-}[ld]&&L\tld{K}_2\ar@{-}[ld]\ar@{-}[rd] && \\
            \tld{K}_1\ar@{-}[rd]&&L\ar@{-}[dd]^G\ar@{-}[ld]_{H_1}\ar@{-}[rd]^{H_2}&&\tld{K}_2\ar@{-}[ld]\\
            &K_1\ar@{-}[rd]&&K_2\ar@{-}[ld]\\
            &&\Q&&
        }
    \]

    From \cref{cor:arith equiv} we deduce that $\tld{K}_1$ and $\al(\tld{K}_1)$ are arithmetically equivalent.
    By \cite[Theorem 1]{arithmetic equivalence Perlis} they have the same Galois closure over $\Q$,
    so $\al(\tld{K}_1) \sub \tld{L}$.
    We also have that $\al(\tld{K}_1)$ is an abelian extension of $K_2$ of exponent $\ell$.
    Consider the number field
    $\tld{K}_2' = \tld{K}_2 \cdot \al(\tld{K}_1)$,
    which is an abelian extension of $K_2$ of exponent $\ell$.
    Since $I_{H_2}\G{\tld{L}/L} = [\G{\tld{L}/K_2}, \G{\tld{L}/L}]$
    acts trivially on both $L$ and $\tld{K}_2'$, we find that
    \[
        I_{H_2}\G{\tld{L}/L} \sub \G{\tld{L}/L\tld{K}_2'} \sub \G{\tld{L}/L\tld{K}_2} = I_{H_2}\G{\tld{L}/L}
    \]
    and thus $L\tld{K}_2' = L\tld{K}_2$.
    Similarly, for $\tld{K}_1' = \tld{K}_1 \cdot \al\inv(\tld{K}_2)$,
    we find that $L\tld{K}_1' = L\tld{K}_1$.
    Note that $\al(\tld{K}_1') = \tld{K}_2'$ by definition,
    and thus applying \cref{cor:arith equiv} to the extensions $\Om_i/\tld{K}_i'$, we find that $\tld{K}_1'$ and $\tld{K}_2'$ are arithmetically equivalent.
    By \cite[Lemma 1]{solvable Neukirch-Uchida},
    arithmetic equivalence is preserved under compositum with a Galois extension of $\Q$,
    so we find that $L\tld{K}_1' = L\tld{K}_1$ and $L\tld{K}_2' = L\tld{K}_2$ are also arithmetically equivalent.
    Therefore, their corresponding subgroups $\G{\tld{L}/L\tld{K}_i} = I_{H_i}\F_\ell[G]^n$
    in $\G{\tld{L}/\Q}$ are Gassmann equivalent \cite{arithmetic equivalence Perlis}.
    In particular, for every element $\lambda \in I_{H_1}\F_\ell[G]^n$
    there must be $g \in G$ such that $g\lambda \in I_{H_2}\F_\ell[G]^n$.

    Enumerate $H_1$ as $h_1, \dots, h_n$
    and consider the element $\lambda = \sum_{1 \le k \le n} (h_k - e)e_k \in I_{H_1}\F_\ell[G]^n$.
    Let $g \in G$ be such that $g\lambda \in I_{H_2}\F_\ell[G]^n$.
    It follows that $\sum_{1 \le k \le n}(gh_k - g)e_k \in I_{H_2}\F_\ell[G]^n$,
    so $gh - g \in I_{H_2}\F_\ell[G]$
    for every $h \in H_1$.
    Thus the elements $gh$ and $g$ identify in the quotient
    $\F_\ell[G]/I_{H_2}\F_\ell[G] \iso \F_\ell[H_2 \backslash G]$,
    so $H_2gh = H_2g$,
    implying that $\cnj{g}{h} \in H_2$. Therefore $\cnj{g}{H_1} \sub H_2$.
    Similarly, there is $k \in G$ such that $\cnj{k}{H_2} \sub H_1$.
    Since $G$ is finite, this implies $H_1$ and $H_2$ are conjugate in $G$.
    This conjugation induces an isomorphism $K_1 \iso K_2$.
\end{proof}

\section{Full Neukirch-Uchida}\label{sec:full-neukirch-uchida}

In this section we present a method to prove the full version of Neukirch-Uchida
using the weak version from \cref{thm:uchida isomorphism}.

Let us outline the proof; we begin by proving a lemma
stating that in the context of our Neukirch-Uchida variant,
with some additional assumption,
there exists an isomorphism $\s:K_1 \to K_2$
whose induced bijection on primes
coincides with the induced bijection of $\al$ 
on primes.
Then, in the general case, we construct arbitrarily large extensions $L_i'/K_i$
for which the condition of the lemma is satisfied.
By applying the lemma to $L_i'$
and considering the group actions on the primes of $L_i'$,
we deduce that $\al$
coincides with a conjugation $g \mapsto \cnj\s{g}$ on arbitrarily large quotients of $\G{\Om_i/K_i}$,
implying the desired result via a compactness argument.

\begin{lem}\label{lem:ell prime differentiation}
    Let $K$ be a number field containing $\mu_\ell$
    and let $\Om/K$ be a Galois extension
    which is $\ell$-sealed with respect to a finite set $S$ of rational places.
    Let $\p_1, \dots, \p_n$ be distinct places of $\Places_\fin(K) \setminus S(K)$
    and let $a_1, \dots, a_n$ be non-negative integers.
    Then there exists a finite subextension $L/K$ of $\Om/K$
    such that $[L_{\fk{P}}:K_{\p_i}] = \ell^{a_i}$ 
    for all $1 \le i \le n$
    and for all $\fk{P} \in \Places_{\p_i}(L)$.
\end{lem}

\begin{proof}
    It suffices to show that for every $1 \le i \le n$ there is
    an extension $L_i/K$ corresponding to the sequence 
    \[0, \dots, 0, a_i, 0, \dots, 0\]
    with $a_i$ at the $i$th place.
    Indeed, if such $L_i$ exist then $L = L_1 \cdot \dots \cdot L_n$ is an extension corresponding to the sequence $a_1, \dots, a_n$.
    Thus we may assume $a_i$ is the only non-zero value in the sequence.

    If $a_i=0$, we can take $L_i = K$.
    Otherwise, by \cref{prop:K^x/K^ell local surjection}
    for the set $S \cup \{\p_1|_\Q, \dots, \p_n|_\Q\}$
    there is $\al \in K$
    whose images in $K_\p, \p \in S(K) \cup \{\p_1, \dots, \p_{i - 1}, \p_{i + 1}, \dots, \p_n\}$
    are $\ell$th powers,
    while its image in $K_{\p_i}$ is not.
    Thus $K' = K(\al^\frac1\ell)$ is a subextension of $\Om$
    in which $\p_1, \dots, \p_{i - 1}, \p_{i + 1}, \dots, \p_n$
    split completely, while $\p_i$ is indecomposable, and $[K'_{\p_i'} : K_{\p_i}] = \ell$
    for $\p_i'$ the unique lift of $\p_i$ to $K'$.
    This reduces the original problem to $K'$
    with the integer sequence
    \[
        0, \dots, 0, a_i - 1, 0, \dots, 0
    \]
    which has length $\ell(n - 1) + 1$
    and in which $a_i - 1$ is in the $\ell(i - 1)+1$ place.
    Applying this argument inductively proves the lemma.
\end{proof}

\begin{rmk}
    Note that the extension $L/K$ constructed above may not be Galois.
\end{rmk}

In the following, we use the notation $\tld{K}$
for the Galois closure of an algebraic extension $K/\Q$.

\begin{lem} \label{lem:alpha identifies with sigma on primes}
    Let $K_1, K_2$ be number fields containing $\mu_\ell$,
    and let $\Om_i/K_i$ be Galois extensions
    that are $\ell$-sealed
    and abundant.
    Suppose $\al: \G{\Om_1/K_1} \to \G{\Om_2/K_2}$ is an isomorphism of profinite groups.
    Assume further that $\tld{K}_2 \cap \Om_2 = K_2$.
    Let $R$ be the set of rational primes with full density
    from \cref{prop:prime bijection}.
    Then there exists an isomorphism $\s:K_1 \to K_2$
    such that the bijection $\s:R(K_1) \to R(K_2)$
    coincides with the bijection $\al:R(K_1) \to R(K_2)$
    from \cref{prop:prime bijection}.
\end{lem}

\begin{proof}
    Since there are finitely many possible isomorphisms $\s:K_1 \to K_2$,
    the final requirement can be changed to
    $\s(\p) = \al(\p)$ for all $\p \in R'(K_1)$,
    where $R'$ is an arbitrary finite subset of $R$.
    Enumerate $R'(K_1)$ as $\p_1, \dots, \p_n$.
    By \cref{lem:ell prime differentiation}, we can find a finite subextension $L_1/K_1$ of $\Om_1/K_1$ 
    such that $[(L_1)_{\fk{P}}:(K_1)_{\p_i}] = \ell^i$ 
    for every $1 \le i \le n$
    and for every $\fk{P} \in \Places_{\p_i}(L_1)$.

    Let $L_2 = \al(L_1)$. By \cref{thm:uchida isomorphism},
    there exists an isomorphism $\s:L_1 \to L_2$.
    Note that $\s(K_1)$ is a subfield of $L_2$
    isomorphic to $K_1$.
    However, $K_1$ is isomorphic to $K_2$ by \cref{thm:uchida isomorphism}.
    It follows that $\s(K_1) \sub \tld{K}_2 \cap \Om_2 = K_2$
    and therefore $\s(K_1) = K_2$.

    Consider an arbitrary $\p_i \in R'(K_1)$
    and $\fk{P} \in \Places_{\p_i}(L_1)$.
    From the fourth part of \cref{prop:prime bijection}
    we have 
    $[(L_2)_{\al(\fk{P})} : (K_2)_{\al(\p_i)}] = [(L_1)_\fk{P} : (K_1)_{\p_i}] = \ell^i$.
    Therefore, 
    $[(L_1)_{\s\inv(\al(\fk{P}))} : (K_1)_{\s\inv(\al(\p_i))}]$
    also equals $\ell^i$.

    The places $\p_i$ and $\s\inv(\al(\p_i))$
    have the same residue characteristic
    by the second part of \cref{prop:prime bijection}.
    In particular $\s\inv(\al(\p_i)) \in R'(K_1)$.
    By the construction of $L_1$,
    which differentiates the places of $R'(K_1)$ by the degrees of their lifts,
    the equality 
    $[(L_1)_{\s\inv(\al(\fk{P}))} : (K_1)_{\s\inv(\al(\p_i))}] = \ell^i$
    is only possible if $\s\inv(\al(\p_i)) = \p_i$,
    i.e. $\s(\p_i) = \al(\p_i)$.
\end{proof}

We can now state and prove our main theorem.

\begin{thm}\label{thm:neukirch uchida}
    Let $\ell$ be a prime number,
    let $K_1, K_2$ be number fields,
    and let $\Om_i/K_i$ be Galois extensions
    that are $\ell$-sealed
    and abundant.
    If $\al:\G{\Om_1/K_1} \to \G{\Om_2/K_2}$ is an isomorphism of profinite groups,
    then there exists an isomorphism of field extensions $\s:\Om_1/K_1 \to \Om_2/K_2$
    that induces $\al$, and it is unique.
\end{thm}

\begin{proof}
    We begin by proving the existence of $\s$.
    Let $L_1/K_1$ be a finite Galois subextension of $\Om_1/K_1$,
    containing $K_1(\mu_\ell) \cdot \al\inv(K_2(\mu_\ell))$,
    and let $L_2 = \al(L_1)$.
    Let $L_2' = \tld{L}_2 \cap \Om_2$
    and $L_1' = \al\inv(L_2')$.
    Note that $\tld{L_2'} \cap \Om_2 \sub \tld{L}_2 \cap \tld{\Om_2} \cap \Om_2 = L_2'$.
    Note also that $L_i'/K_i$ are Galois and contain $\mu_\ell$.
    By \cref{lem:alpha identifies with sigma on primes},
    there is an isomorphism $\s:L_1' \to L_2'$
    such that $\s(\fk{P}) = \al(\fk{P})$
    for every place $\fk{P} \in R(L_1')$
    for a set $R$ of rational primes with full density.
    By Chebotarev's Density Theorem,
    there is a prime number $p \in R$
    that splits completely in $L_2'$.
    Then for every $g \in \G{L_1'/K_1}$ 
    and $\fk{P} \in \Places_p(L_2')$
    we have
    \[
        (\cnj\s{g})\fk{P} = \s(g(\s\inv(\fk{P}))) = \al(g(\al\inv(\fk{P}))) = \al(g)\fk{P}
    \]
    as primes in $L_2'$.
    Since $L_2'/K_2$ and $L_2'/\s(K_1)$ are both Galois,
    we find that $L_2'/(K_2 \cap \s(K_1))$ is Galois.
    Since $\G{L_2'/(K_2 \cap \s(K_1))}$ acts freely on $\Places_p(L_2')$,
    the above shows that for every $g \in \G{L_1'/K_1}$, the equality $\cnj\s{g} = \al(g)$ holds in $\G{L_2'/(K_2 \cap \s(K_1))}$.
    In particular, we have 
    \[
        \G{L_2'/K_2} = \al(\G{L_1'/K_1}) = \cnj{\s}{\G{L_1'/K_1}} = \G{L_2'/\s(K_1)}
    \]
    as subgroups of $\G{L_2'/(K_2 \cap \s(K_1))}$. This implies $\s(K_1) = K_2$.
    Likewise, we have
    \[
        \G{L_2'/L_2} = \al(\G{L_1'/L_1}) = \cnj{\s}{\G{L_1'/L_1}} = \G{L_2'/\s(L_1)}
    \]
    as subgroups of $\G{L_2'/K_2}$,
    implying that $\s(L_1) = L_2$.
    Note that the conjugation $\cnj{\s}{(-)}:\G{L_1/K_1} \to \G{L_2/K_2}$
    coincides with
    $\al : \G{L_1/K_1} \to \G{L_2/K_2}$.

    For every finite Galois subextension $L_1/K_1$
    of $\Om_1/K_1$, 
    we define $C_{L_1} \sub \G\Q$
    as the set of $\s\in\G\Q$
    such that $\s(K_1) = K_2$,
    $\s(L_1) = \al(L_1)$,
    and the conjugation
    $\cnj{\s}{(-)} : \G{L_1/K_1} \to \G{\al(L_1)/K_2}$
    coincides with $\al : \G{L_1/K_1} \to \G{\al(L_1)/K_2}$.
    Then $C_{L_1}$ is a closed subset of $\G\Q$,
    and was just shown to be nonempty when $L_1$ contains $K_1(\mu_\ell) \cdot \al\inv(K_2(\mu_\ell))$.
    It is clear that $C_{L_1} \cap C_{L_1'} = C_{L_1L_1'}$.
    Thus, the compactness of $\G\Q$
    ensures that there exists an element $\s$ of $\bigcap_{L_1} C_{L_1}$,
    where the intersection runs over all finite Galois subextensions of $\Om_1/K_1$.
    It follows that $\s(K_1) = K_2$,
     $\s(\Om_1) = \Om_2$,
    and also that the conjugation map $\cnj{\s}{(-)} : \G{\Om_1/K_1} \to \G{\Om_2/K_2}$
    coincides with $\al : \G{\Om_1/K_1} \to \G{\Om_2/K_2}$.
    This completes the existence proof.

    For uniqueness,
    we may assume $K_1 = K_2$ and $\Om_1=\Om_2$,
    denoting them by $K$ and $\Om$.
    Let $S$ be a finite set of rational primes
    such that $\Om/K$ is $\ell$-sealed with respect to it.
    Let $\s \in \G{\Om/K}$ be an element
    such that conjugation by it induces the identity automorphism of $\G{\Om/K}$,
    i.e. $\s \in Z(\G{\Om/K})$.
    Let $L$ be a finite Galois subextension of $\Om/K$
    containing $\mu_\ell$.
    Since $\G{\Om/\s(L)} = \cnj\s{\G{\Om/L}}= \G{\Om/L}$,
    we have $\s(L) = L$.
    For every $\fk{P} \in \Places_\fin(\Om) \setminus S(\Om)$ we have
    \[
        \G{\Om/K, \fk{P}} = \cnj\s{\G{\Om/K, \fk{P}}} = \G{\Om/K, \s(\fk{P})}
    \]
    implying that $\s$ acts as the identity on $\Places_\fin(\Om) \setminus S(\Om)$
    by \cref{lem:Xi/ell=0 implies ell^infty | [Xi:F]} and \cref{prop:local groups ell-disjoint}.
    Let $B \sub L$ be the fixed field of $\s$ in $L$.
    By Chebotarev's Density Theorem,
    there is $\p \in \Places_\fin(B) \setminus S(B)$
    that splits completely in $L$.
    Since $\G{L/B} = \ang{\s|_L}$ acts transitively on $\Places_\p(L)$
    while $\s$ acts trivially on it, 
    we must have $B = L$, i.e. $\s|_L = \id_L$.
    Since $L$ is an arbitrarily large finite subextension of $\Om/K$,
    it follows that $\s$ is trivial.
    This concludes the proof.
\end{proof}

\begin{rmk}
    The uniqueness part of the above proof
    did not use the abundance assumption.
\end{rmk}

From \cref{ex:ell-sealed-examples} and \cref{ex:abundance-examples},
we find that the extensions $\Om_i/K_i$
described in both \cref{thm:pro-ell-by-cyclotomic neukirch uchida} and \cref{thm:tamely ramified neukirch uchida}
are $\ell$-sealed and abundant.
Thus both theorems follow from \cref{thm:neukirch uchida}.

\section{Acknowledgments}

The authors' research is co-funded by the European Union (ERC, Function Fields, 101161909). Views and opinions expressed are however those of the authors only and do not necessarily reflect those of the European Union or the European Research Council. Neither the European Union nor the granting authority can be held responsible for them.



\begin{thebibliography}{99}

    \bibitem{solvable Neukirch-Uchida}
    Uchida, K. (1979). Isomorphisms of Galois groups of solvably closed Galois extensions. Tohoku Mathematical Journal, Second Series, 31(3), 359-362.

    \bibitem{3-solvable Neukirch-Uchida}
    Saidi, M., \& Tamagawa, A. (2022). The m-step solvable anabelian geometry of number fields. Journal für die reine und angewandte Mathematik (Crelles Journal), 2022(789), 153-186.

    \bibitem{restricted ramification Neukirch-Uchida}
    Shimizu, R. (2020). The Neukirch-Uchida theorem with restricted ramification. arXiv preprint arXiv:2009.10431.

    \bibitem{restricted ramification Neukirch-Uchida 2}
    Shimizu, R. (2023). Isomorphisms of Galois groups of number fields with restricted ramification. Mathematische Nachrichten, 296(7), 3026-3033.

    \bibitem{restricted ramification Neukirch-Uchida with stable sets}
    Ivanov, A. (2013). On a generalization of the Neukirch-Uchida theorem. arXiv preprint arXiv:1309.3046.

    \bibitem{p-closed almost everywhere Neukirch-Uchida}
    Sueyoshi, Y. (1988). A characterization of number fields by p-closed extensions. Journal of Number Theory, 28, 145-151.
    
    \bibitem{unramified-by-cyclotomic Neukirch-Uchida}
    Ozaki, M. (2025). A number field analogue of the Grothendieck conjecture for curves over finite fields. arXiv preprint arXiv:2507.20159.

    \bibitem{pro-C Neukirch-Uchida}
    Shimizu, R. (2023). The pro-C anabelian geometry of number fields.

    \bibitem{2-nilpotent anti-Neukirch-Uchida}
    Koymans, P., \& Pagano, C. (2024). Two-step nilpotent extensions are not anabelian. Mathematische Zeitschrift, 306(2), 31.
    
    \bibitem{nilpotent anti-Neukirch-Uchida}
    Golich, M., \& McReynolds, D. B. (2023). Diamonds: Homology and the Central Series of Groups. arXiv preprint arXiv:2310.08283.

    \bibitem{sylow anti-Neukirch-Uchida}
    Lubotzky, A., \& Neftin, D. (2023). Sylow-conjugate number fields. Israel Journal of Mathematics, 257(2), 465-480.
    
    \bibitem{cohomology of number fields}
    Neukirch, J., Schmidt, A., \& Wingberg, K. (2013). Cohomology of number fields (Vol. 323). Springer Science \& Business Media.

    \bibitem{algebraic number theory Neukirch}
    Neukirch, J. (2013). Algebraic number theory (Vol. 322). Springer Science \& Business Media.

    \bibitem{profinite groups ribes zalesskii}
    Ribes, L., \& Zalesskii, P. (2010). Profinite groups. In Profinite Groups (pp. 19-74). Berlin, Heidelberg: Springer Berlin Heidelberg.

    \bibitem{arithmetic equivalence Perlis}
    Perlis, R. (1977). On the equation $\zeta_k (s)= \zeta_{k'}(s)$. Journal of number theory, 9(3), 342-360.

    \bibitem{arithmetic equivalence by number of primes}
    Stuart, D., \& Perlis, R. (1995). A new characterization of arithmetic equivalence. Journal of Number theory, 53(2), 300-308.

\end{thebibliography}
\end{document}